\documentclass[a4paper, 12pt]{article}

\usepackage[utf8]{inputenc}
\usepackage[T1]{fontenc}
\usepackage[a4paper, margin=2.5cm]{geometry}
\usepackage{needspace}
\usepackage{svg}
\usepackage{libertine} 
\usepackage{inconsolata} 
\usepackage{amsmath, amsthm, amssymb}
\usepackage{graphicx}
\usepackage{enumerate}
\usepackage{authblk}
\usepackage[textsize=small, textwidth=2cm, color=yellow]{todonotes}
\usepackage{thmtools}
\usepackage[colorlinks=true, linkcolor =blue!80, citecolor=orange!70!black]{hyperref}
\usepackage{float}
\usepackage[capitalize]{cleveref}
\usepackage{comment}
\usepackage{tikz}
\usepackage{multirow}
\usepackage{thm-restate}
\usepackage{soul}

\renewenvironment{abstract}
{\small\vspace{-1em}
\begin{center}
\bfseries\abstractname\vspace{-.5em}\vspace{0pt}
\end{center}
\list{}{
\setlength{\leftmargin}{0.6in}%
\setlength{\rightmargin}{\leftmargin}}%
\item\relax}
{\endlist}

\Crefname{observation}{Observation}{Observations}

\declaretheorem[name=Theorem, numberwithin=section]{theorem}
\declaretheorem[name=Lemma, sibling=theorem]{lemma}

\declaretheorem[name=Corollary, sibling=theorem]{corollary}

\declaretheorem[name=Problem, sibling=theorem]{problem}

\declaretheorem[name=Remark, style=remark, sibling=theorem]{remark}

\def\cqedsymbol{\ifmmode$\lrcorner$\else{\unskip\nobreak\hfil
\penalty50\hskip1em\null\nobreak\hfil$\lrcorner$
\parfillskip=0pt\finalhyphendemerits=0\endgraf}\fi}

\DeclareMathOperator{\td}{\mathrm{td}}
\DeclareMathOperator{\tw}{\mathrm{tw}}
\DeclareMathOperator{\pw}{\mathrm{pw}}

\DeclareMathOperator{\poly}{poly}

\let\le\leqslant
\let\ge\geqslant
\let\leq\leqslant
\let\geq\geqslant

\renewcommand{\mid}{:}

\title{Faithful universal graphs for minor-closed classes}

\author[1]{Paul Bastide}
\author[2]{Louis Esperet\thanks{Partially supported by the French ANR Projects TWIN-WIDTH
  (ANR-21-CE48-0014-01) and ENEDISC (ANR-24-CE48-7768), and by LabEx
  PERSYVAL-lab (ANR-11-LABX-0025).}}
\author[3]{Carla Groenland\thanks{Supported by NWO grant VI.Veni.232.073.}}
\author[4]{Claire Hilaire\thanks{Partially supported by Slovenian Research and Innovation Agency (research project J1-4008).}}
\author[5]{Clément Rambaud\thanks{Partially supported by the ANR project DIGRAPHS (ANR-19-CE48-0013)}}
\author[6]{Alexandra Wesolek\thanks{ Supported by the DFG under Germany’s Excellence Strategy – The Berlin Mathematics Research Center MATH+ (EXC-2046/1, project ID: 390685689).}}

\affil[1]{Université de Bordeaux, LaBRI, Bordeaux, France}
\affil[2]{Université Grenoble Alpes, CNRS, Laboratoire G-SCOP,
  Grenoble, France}
\affil[3]{Delft Institute of Applied Mathematics, Technische Universiteit Delft, The Netherlands}
\affil[4]{Université Clermont Auvergne, Clermont Auvergne INP, LIMOS, Clermont-Ferrand, France}

\affil[5]{Université Côte d'Azur, I3S, Inria, CNRS, Sophia-Antipolis, France}

\affil[6]{Technische Universität Berlin, Berlin, Germany}

\date{\today}

\begin{document}

\maketitle

\begin{abstract}
It was proved by Huynh, Mohar, Šámal, Thomassen and Wood in 2021 that any countable graph containing every countable planar graph as a subgraph has an infinite clique minor. We prove a finite, quantitative version of this result: for fixed $t$, if a graph $G$ is $K_t$-minor-free and contains every $n$-vertex planar graph as a subgraph, then $G$ has $2^{\Omega(n)}$ vertices.  On the other hand, we construct a polynomial size $K_4$-minor-free graph containing every $n$-vertex tree as an induced subgraph, and a polynomial size $K_7$-minor-free graph containing every $n$-vertex $K_4$-minor-free graph as induced subgraph. This answers several problems raised recently by Bergold, Ir\v{s}i\v{c}, Lauff, Orthaber, Scheucher and Wesolek.

We study more generally the order of universal graphs for various classes (of graphs of bounded degree, treedepth, pathwidth, or treewidth), if the universal graphs retain some of the structure of the original class.
\end{abstract}

\section{Introduction}

A family of graphs $\mathcal{U}=(U_n)_{n\in \mathbb{N}}$ is said to be \emph{induced-universal} for a graph class $\mathcal{G}$ if for any $n\in \mathbb{N}$, $U_n$ contains all the $n$-vertex graphs of $\mathcal{G}$ as induced subgraphs. Similarly, $\mathcal{U}$ is said to be \emph{subgraph-universal} for $\mathcal{G}$ if for any $n\in \mathbb{N}$, $U_n$ contains all the $n$-vertex graphs of $\mathcal{G}$ as subgraph. 

The initial motivation for the introduction of these concepts was the design of configurable chips \cite{BL82,CRS83}, with the idea that a single chip could be used to produce a number of different chips, by simply removing connectors (edges) or components (vertices) in a post-processing phase. This lead to a considerable amount of work, trying to either minimise the number of edges in subgraph-universal graphs \cite{chung1990-separator,chung.graham,bhatt.chung.ea, babai.chung.ea,Val81,bhatt1986optimal,Capalbo02,ACKRRS,AC08,EJM23}, or minimise the number of vertices in induced universal graphs \cite{chung1990-universal-induced-universal, gavoille.labourel:shorter,bonamy.gavoille.ea:shorter, AdjacencyLabellingPlanarJACM,kannan.naor.ea:implicit, Alon17,spinrad:efficient,muller:local,alstrup.dahlgaard.ea:optimal,alstrup.kaplan.ea:adjacency,abrahamsen.alstrup.ea:near-optimal,AlonCapalbo,EJM23}. 

For induced-universal graphs, we give a sample of the best-known  results below. By a slight abuse of notation, when we say that for some function $f$, \emph{a graph class $\mathcal{G}$ has  subgraph-universal (or induced-universal) graphs of order at most $n\mapsto f(n)$}, we mean that $\mathcal{G}$ has a subgraph-universal (or induced-universal) graph family $\mathcal{U}=(U_n)_{n\in \mathbb{N}}$ such that each $U_n$ contains at most $f(n)$ vertices. With this terminology, there are induced-universal graphs of order 
\begin{itemize}
\item $O(n)$ for trees \cite{alstrup.dahlgaard.ea:optimal},
\item $n^{1+o(1)}$ for graphs of bounded treewidth \cite{gavoille.labourel:shorter},
\item $n^{1+o(1)}$ for planar graphs and a number of related graph classes \cite{AdjacencyLabellingPlanarJACM,GJ22,EJM23},
 \item $n^{2+o(1)}$ for proper minor-closed classes \cite{gavoille.labourel:shorter},
 \item $O(n^{\Delta/2})$ for the class of graphs of maximum degree $\Delta$ \cite{AN19}, and
    \item $(1+o(1))2^{(n-1)/2}$ for the class of all graphs \cite{Alon17}.
    \end{itemize}
A classical problem related to induced-universal graphs was the \emph{Implicit graph conjecture} (see \cite{kannan.naor.ea:implicit,spinrad:efficient}), which stated that every hereditary class of graphs which contains at most $2^{O(n\log n)}$ graphs with $n$ vertices has an induced-universal graph of polynomial size. This conjecture was only refuted recently by Hatami and Hatami \cite{HH21}.

\medskip
   
    Most of the results mentioned above are optimal, or close to optimal. The main techniques for constructing induced-universal graphs use the notion of an \emph{adjacency labelling scheme}, introduced by Kannan, Naor and Rudich \cite{kannan.naor.ea:implicit} and Muller \cite{muller:local}. This is a compact data structure describing the adjacency relation between vertices in a graph class. In the translation between the data structure and the universal graph, all the structure of the original graphs is lost and even controlling the number of edges is difficult (see for instance \cite{EJM23}).  

\medskip
    
    Note that without further constraints, it is easy to minimise the number of \emph{vertices} in a subgraph-universal graph for any family: just take $U_n$ to be the complete graph on $n$-vertices (it certainly contains all the $n$-vertex graphs as subgraphs, and cannot be made smaller). So a more interesting goal is to minimise the number of \emph{edges} in a subgraph-universal graph. The best known bounds on the number of edges of subgraph-universal graphs usually come from separator theorems \cite{chung1990-separator}. If the vertex set of any $n$-vertex graph $G\in \mathcal{G}$ can be partitioned into three sets $X_1,S,X_2$, with $|X_1|\le 2n/3$, $|X_2|\le 2n/3$, $|S|=O(n^{1-\varepsilon})$, and no edges between $X_1$ and $X_2$, then the idea is to take (inductively) two subgraph-universal graphs for the graphs of size $|X_1|$ and $|X_2|$ in the class, and add a set of $|S|$ universal vertices. This typically creates large cliques, even if the graphs from  $\mathcal{G}$ have small clique number.

    \paragraph{Faithful universal graphs}
    
    Let us say that a universal graph family $\mathcal{U}=(U_n)_{n\in \mathbb{N}}$ for a graph class $\mathcal{G}$ is \emph{faithful} if for any $n\in \mathbb{N}$, $U_n\in \mathcal{G}$. In the context of chip design, the configurable chip has the same physical constraints as the different chips it is supposed to emulate, so in graph-theoretic terms we expect that the underlying (universal) graph is faithful (for instance, if all graphs underlying the different chips have to be planar, the configurable chip also has to be planar). As suggested above, all the techniques mentioned in the previous paragraph produce universal graphs that are not faithful. Moreover, in the few cases where some non-trivial faithful universal graphs are known, their size is significantly larger than the size of the optimal universal graphs.

     \smallskip

    Consider for instance the case of trees. It was proved by Chung and Graham \cite{chung.graham} that there is a subgraph-universal graph with $\Theta(n\log n)$ edges for trees\footnote{See also \cite{FHT23,KKKW25}, who identified a mistake in the original proof of Chung and Graham and provided corrected versions.},
    and that this bound is best possible. As explained above, the proof uses separators and produces graphs with cliques of logarithmic size, so the resulting graphs are very far from trees (or any other class of sparse graphs). On the other hand, Gol'dberg and Livshits  \cite{GL68} constructed  faithful subgraph-universal graphs for trees on $2^{O(\log^2n)}$ vertices, and this order of magnitude was later shown to be asymptotically best possible \cite{CGC81}. Note that as explained above, while minimising the number of vertices in a subgraph-universal in general is not relevant (because complete graphs contain all graphs of the same order as subgraphs), minimising the number of vertices in a \emph{faithful} subgraph-universal graph is interesting, especially when the original graph class $\mathcal{G}$ is sparse (and thus does not contain large cliques). Note also that while there is no tree of polynomial size containing all trees on $n$ vertices as subgraph, there is a binary tree of size $O(n^4)$ containing all binary trees on $n$ vertices as minor \cite{HWY10}, and this result has interesting applications in algebraic complexity theory in the presence of  non-associativity.

    \paragraph{Countable universal graphs} Beyond the motivation of designing configurable chips, the concept of faithful universal graphs is well-established  in mathematics (though under a different name) and has received a lot of attention over the past 40 years in the context of infinite graphs \cite{ackermann1937widerspruchsfreiheit,komjath1984universal,diestel1985universal,komjath1988some,furedi1997nonexistence,furedi1997existence,komjath1999some,cherlin1999universal,cherlin2001forbidden,cherlin2007universalforbidden,cherlin2007universal,cherlin2016universal,huynh2021universality}. In this context, we say that an (infinite) countable graph \( U \) is \emph{faithful induced-universal} for a class of countable graphs \( \mathcal{G} \) if \( U \in \mathcal{G} \) and \( U \) contains every graph from \( \mathcal{G} \) as an induced subgraph. A well-known example of this notion is the \emph{Rado graph}. Ackermann \cite{ackermann1937widerspruchsfreiheit}, Erd\H{o}s and Rényi \cite{erdos1963asymmetric}, and Rado \cite{rado1964universal} proved that the Rado graph is faithful induced-universal for the class of all countable graphs. The Rado graph has, since then, proven to be useful in solving multiple other questions in combinatorics (see the following surveys \cite{cameron1997random,cameron2001random}). Henson \cite{henson1971family} has proved the existence of a countable faithful induced-universal graph for the class of countable $K_t$-subgraph-free graphs. Since then, there has been a systematic study of faithful universal graphs for the family of countable $H$-free graphs \cite{komjath1984universal,komjath1988some,furedi1997nonexistence,furedi1997existence,komjath1999some,cherlin1999universal,cherlin2001forbidden,cherlin2007universalforbidden,cherlin2007universal,cherlin2016universal}. Diestel, Halin and Vogler \cite{diestel1985some}
proved that for any $t \leq 4$, there does not exist a faithful induced-universal graphs for the family of countable $K_t$-minor-free graphs, and similar results are known when excluding $K_{s,t}$ as a minor \cite{diestel1985universal}.

\medskip

In the context of minor-closed classes of graphs, the following question was raised by Ulam:
\begin{problem}
    Does there exist a countable faithful subgraph-universal graph for the class of countable planar graphs?
\end{problem}
This question was answered negatively by Pach \cite{pach1981problem}. In his paper, Pach also asked whether the  condition of being faithful could be relaxed to obtain a countable subgraph-universal graph for the class of countable planar graphs while still preserving key properties of planar graphs. Following the direction of Pach's question, it is natural to ask whether there exists a subgraph-universal graph $U$ for the class of countable planar graphs, such that $U$ is $H$-minor-free for some graph $H$. Forty years after Pach's results, Huynh, Mohar, Šámal, Thomassen and Wood \cite{huynh2021universality} provided a negative answer to Pach's question. 

\begin{theorem}[\cite{huynh2021universality}]\label{thm:h21}
    If $U$ is a countable subgraph-universal graph for the class of countable planar graphs, then the countable complete graph $K_{\aleph_0}$ is a minor of $U$.
\end{theorem}

\paragraph{Finite universal graphs}
From now on, and throughout the rest of the paper, all graphs will be considered finite unless stated otherwise. Ulam’s question, as well as Pach’s follow-up question, can also be considered in the finite setting. Note that for any class of finite graphs \(\mathcal{G}\) that is closed under disjoint union, there exists a very simple family \((U_n)_{n \in \mathbb{N}}\) of faithful induced-universal graphs for \(\mathcal{G}\), which consists of the disjoint union of all \(n\)-vertex graphs in \(\mathcal{G}\). Observe that this solution is far from satisfactory, as the size of \(U_n\) grows quickly with $n$. For instance, in the case of planar graphs, such a disjoint union would have size exponential in $n$. The following is a natural reformulation of Ulam's question in the finite  setting, where \emph{polynomial} plays the role of \emph{countable}, and \emph{exponential} plays the role of \emph{uncountable}.

\begin{problem}
    \label{prob:ulam_finite}
    Is there a constant $k$ such that for any integer $n$,  there exists a planar graph $U_n$ on $n^{k}$ vertices which contains every $n$-vertex planar graph as a subgraph?
\end{problem}

    It was recently proved by Bergold, Ir\v{s}i\v{c}, Lauff, Orthaber, Scheucher and Wesolek \cite{BILOSW24} that any faithful subgraph-universal graph for the class of outerplanar graphs has exponential size, and similarly any faithful subgraph-universal graph for the class of planar graphs has exponential size, giving  a negative answer to Problem \ref{prob:ulam_finite}. On the other hand, they showed that there is a subgraph-universal family of outerplanar graphs of polynomial order for the class of trees, and a subgraph-universal family of planar graphs of  order $n\mapsto 2^{O(\log^2 n)}$ for the class of outerplanar graphs. They raised the following problem, which can be viewed as a finite version of Pach’s question mentioned above.

\begin{problem}[\cite{BILOSW24}]\label{pro:1}
Is it true that for any $t\in \mathbb{N}$, there is a family of $K_{t+1}$-minor-free subgraph-universal graphs of  polynomial order for the class of $K_t$-minor-free graphs?
\end{problem}
More generally, we can ask whether we can  find universal graphs of reasonable order (say polynomial) that retain some of the structure of the original graph class. 

\begin{problem}\label{pro:2}
Is it true that for any graph $H$, there is a graph $H'$ and a  family of $H'$-minor-free subgraph-universal graphs of polynomial order for the class of $H$-minor-free graphs?
\end{problem}
    The authors of \cite{BILOSW24} also raised the problem of extending their constructions to induced-universal graphs, in particular:

    \begin{problem}[\cite{BILOSW24}]\label{pro:3}
Is there a family of planar induced-universal graphs of polynomial order for the class of all trees?
\end{problem}

\subsection{Our results}

In this paper we prove lower and upper bounds on the order of universal graphs for various well-studied classes (of bounded degree, treedepth, pathwidth, treewidth, or $K_t$-minor-free graphs) if the universal graphs are required to be faithful, or only "approximately" faithful.  Note that lower bounds on the order of subgraph-universal graphs hold for induced-universal graphs as well, while upper bounds for induced-universal graphs hold for subgraph-universal graphs as well. Our main results are the following:

\begin{theorem}\label{thm:intro1}
    Any $K_t$-minor-free graph containing all planar $n$-vertex graphs as subgraph has order at least $2^{\Omega_t(n)}$.
\end{theorem}

This can be seen as a finite version of Theorem \ref{thm:h21}. 
Theorem \ref{thm:intro1} provides a strong negative answer to Problem \ref{pro:2} and to Problem \ref{pro:1} for $t\ge 5$. On the other hand,  we prove  the following positive results, which have direct consequences for $t\le 4$.
\begin{theorem}\label{thm:intro3}
For any integer $k\ge 1$, there is a constant $c_k$ such that for any $n$, there is a graph of treewidth $3k-1$ on $n^{c_k}$ vertices which contains all $n$-vertex graphs of treewidth $k$ as an induced subgraph. In particular,
\begin{itemize}
\item there is a $K_4$-minor-free graph of polynomial order which contains every $n$-vertex tree as an induced subgraph, and
    \item there is a $K_7$-minor-free graph of polynomial order which contains every $n$-vertex $K_4$-minor-free graph as an induced subgraph.   
\end{itemize}    
\end{theorem}
In particular, since both $K_{3,3}$ and $K_5$ have a $K_4$-minor, the result above shows that there is a planar graph of polynomial order that contains every tree as induced subgraph, providing a positive answer to Problem~\ref{pro:3}.

\smallskip

It follows from the results above that for any proper minor-closed class $\mathcal{G}$ the following holds.

\begin{itemize}
    \item If $\mathcal{G}$ contains all planar graphs, then for any fixed proper minor-closed class $\mathcal{G}'$, any  graph from $\mathcal{G}'$ containing all $n$-vertex graphs from $\mathcal{G}$ as subgraph has size $2^{\Omega(n)}$.
    \item If $\mathcal{G}$ excludes some planar graph, then $\mathcal{G}$ has bounded treewidth and thus 
    there is a proper minor-closed class $\mathcal{G}'$ of bounded treewidth such that for all $n$, there is a graph from $\mathcal{G}'$ of size polynomial in $n$ that contains all $n$-vertex graphs from $\mathcal{G}$ as induced subgraph.
\end{itemize}

The remainder of our results is summarized in the following tables.

\begin{table}[H]\label{ta:1}
\begin{center}
\begin{tabular}{|c|c||c|c|}
    \hline Class  &  Exact &  \multicolumn{2}{c|}{Approximation} \\
    \hline \hline  treedepth $k \geq 1$ & $O(n^{k-1})$ (\ref{td upper bound})& -- & -- \\
    \hline pathwidth $k \geq 1$ & $2^{\Omega(n\log k)}$ (\ref{cor:lmtw}) for $k\ge 2$ & pathwidth $k^2+k-1$ & $\le n^{k+1}$ (\ref{pw upper bound})\\
    \hline treewidth $k \geq 1$ & $2^{\Omega(n\log k)}$ (\ref{cor:lmtw}) for $k\ge 2$ & treewidth $3k-1$ & $n^{O_k(1)}$ (\ref{tw upper bound}) \\
        \hline max. degree $\Delta\ge 3$ & --  & max. degree $n^{o(1)}$ & $\ge n^{(\Delta/2-1-o(1)) n}$ (\ref{thm:degree})\\
\hline
\end{tabular}
\caption{Results for \textbf{subgraph-universal}  graphs.}
\end{center}
\end{table}

\begin{table}[H]\label{ta:2}
\begin{center}
\begin{tabular}{|c||c|c|}
   \hline Class  &   \multicolumn{2}{c|}{Approximation} \\
    \hline \hline treedepth $k \geq 1$ &  treedepth $k\cdot 2^k$ & $O(2^{k} n^{k-1})$  (\ref{cor: td upper bound iu})\\
    \hline pathwidth $k \geq 2$  & pathwidth $2^{O(k^2)}$ & $2^{O(k^2)}n^{k+1}$ (\ref{cor: pw upper bound})\\
    \hline treewidth $k \geq 1$  & treewidth $3k-1$ & $n^{O_k(1)}$ (\ref{tw upper bound}) \\
\hline
\end{tabular}
\caption{Results for \textbf{induced-universal}  graphs.}
\end{center}
\end{table}

\begin{table}[H]
\begin{center}
\begin{tabular}{|c|c|c|}
   \hline  $k$  &  $t$ &  Order  \\
    \hline \hline \multirow{3}*{4} & 4  & $2^{\Omega(n)}$ (\ref{cor:lmtw}) \\
    \cline{2-3}  & 5 & $n^{O(\log n)}$ (\ref{thm:K5-m-free universal for K4-m-free}) \\
    \cline{2-3}  & 7 & $n^{O(1)}$ (\ref{cor:poly upper bound2}) \\
    \hline $k\ge 5$ & $k$ & $\geq 2^{(n-k-2)/(k-2)}$ (\ref{lem:cor_kt_lowerbound}) \\
    \hline $k\ge 5$ & $t\ge k$ & $2^{\Omega(n/\poly(t))}$ (\ref{thm:Kt free main})\\\hline
\end{tabular}
\label{table:3}
\caption{Bounds on the order of a $K_t$-minor-free graph containing all $n$-vertex $K_k$-minor-free graphs as subgraph. Note that the upper bound for $k=4$ and $t=7$ holds for induced subgraphs as well. }
\end{center}
\end{table}
In~\cite{BILOSW24}, the lower bound $n\mapsto n^{-3/2}(27/4)^{n}$  is obtained for the order of a faithful subgraph-universal graph for the class of planar graphs and the suggestion was made that the proof technique may extend to $K_t$-minor-free graphs. We in fact generalise the proof technique of another lower bound from this paper (the lower bound of  $n\mapsto 2^{(n-5)/2}$ for a subgraph-universal outerplanar graph for pathwidth 2 graphs) to show that any graph of treewidth $k$ containing all edge-maximal $n$-vertex graphs of pathwidth $k$ has size $2^{\Omega(n\log k)}$.
We provide a lemma (\cref{lem:universal_vx_removal_lemma}) that allows us to add a universal vertex to $H$ in a lower bound for faithful $H$-minor-free universal graphs and use this to obtain a lower bound of $n\mapsto 2^{\Omega_t(n)}$ for the order of faithful subgraph-universal graphs for the class of $K_t$-minor-free graphs.
The proof also yields an exponential lower bound for faithful subgraph-universal planar graphs but with a slightly worse constant than that of \cite{BILOSW24}. 

These lower bounds are (nearly) best possible. For example, the disjoint union of the edge-maximal $n$-vertex graphs of pathwidth $k$ has size $2^{O(n\log k)}$ (and contains all $n$-vertex graphs of pathwidth $k$ as subgraph) and similarly the disjoint union of $n$-vertex $K_t$-minor-free graphs has size at most $2^{O_t(n)}$ (and is of course induced-universal for that class). This matches our lower bounds up to the constant term in the exponent. 

\smallskip

For the class of graphs of maximum degree $\Delta$, our lower bound $n^{(\Delta/2-1-o(1))n}$  (see also Theorem~\ref{thm:degree}) is very nearly best possible, since a disjoint union of $n$-vertex graphs of maximum degree $\Delta$ would give an induced-universal graph of order $n^{(\Delta/2+o(1))n}$ with maximum degree $\Delta$. This lower bound also applies when we only require the universal graph to be subgraph-universal and allow a much larger maximum degree (smaller than any polynomial in $n$). 

\smallskip

All of our (upper bound) constructions are relatively simple. The induced-universal graph of treewidth~$3k-1$ for graphs of treewidth~$k$ uses the fact that it is possible to find a tree-decomposition of width $3k-1$ of logarithmic depth for every graph of treewidth~$k$. We then create a ``universal tree-decomposition'' for such tree-decompositions with logarithmic depth and degree $O_k(1)$, which has therefore $n^{O_k(1)}$ vertices. The subgraph-universal construction for pathwidth is an inductive pathwidth-specific construction. We provide a lemma (Lemma~\ref{lem:sub to ind}) to turn degenerate subgraph-universal graphs into (not much larger) induced-universal graphs that we apply to the class of graphs of bounded pathwidth. This lemma is based on the ``standard'' adjacency  labelling scheme for degenerate graphs. Finally, using ideas from~\cite{BILOSW24}, we also give a subgraph-universal graph of treewidth 3 (so $K_5$-minor-free) of quasi-polynomial order for treewidth 2 (i.e. $K_4$-minor-free) graphs. 

\subsection*{Proof technique of the main result}
To prove our main result (Theorem~\ref{thm:intro1}), we show in Theorem~\ref{thm:Kt free main} that there is a polynomial function $p(t)$ of $t$ such that any $K_t$-minor-free graph containing all $n$-vertex planar graphs as subgraph has order at least $2^{\Omega(n/p(t))}$. We follow the approach used by Huynh, Mohar, Šámal, Thomassen and Wood \cite{huynh2021universality} in the infinite case.
If a universal graph $U$ contains all planar graphs, then in particular, it contains all triangulated grids. Note that while in the infinite case there is an uncountable number of such triangulated grids, in the finite case this number is exponential. We then try to find a triangulated grid $G$ in $U$ where there are many vertices in $G$ that are connected in $U$ but not in $G$ (which we call \textit{jumps}). If we have $\Omega_t(1)$ jumps, then the jumps plus $G$ create a $K_t$-minor in $U$ (see Lemma~\ref{Kt double grid one interior bis} for a formal statement). 

To find the jumps, we note that when many triangulated grids can be embedded on a relatively small vertex set, then certainly two grids $G_1$ and $G_2$ must intersect in a ``non-trivial manner''. 
For example, suppose that $G_1$ and $G_2$ are grids embedded in $U$ and there is a path between $u,v\in V(G_1)\cap V(G_2)$ with all internal vertices in $V(G_1)\setminus V(G_2)$. This gives rise to a ``jump'' in $G_2$, except when $uv\in E(G_2)$. But if $uv\in E(G_2)$, then we can use this as a ``jump'' in $V(G_1)$ if we choose $u$ and $v$ such that $uv\not\in E(G_1)$. So our aim is to create many such paths in $G_1$ that are all vertex-disjoint.

Using a VC-dimension argument, it is possible to find triangulated grids $G_1$ and $G_2$ embedded into $U$ with $|V(G_1)\cap V(G_2)|$ of ``medium'' size. However, this is not sufficient: for example, $G_1$ and $G_2$ could be embedded on a sphere with the  intersection along their boundaries without creating a $K_t$-minor. Instead, we use a VC-dimension argument to control which vertices are in $G_2$ for a set of vertices of size $d$ in $G_1$. We use this to create many jumps in either $G_1$ or $G_2$, eventually leading to a $K_t$-minor.

\subsection*{Organisation of the paper} 
We start with some preliminary definitions and results in Section \ref{sec:prel}. In Section \ref{sec:poly} we construct small universal graphs for graphs of bounded treedepth, pathwidth, and treewidth, proving in particular Theorem \ref{thm:intro3}. In Section \ref{sec:lb} we prove lower bounds on the order of universal graphs which are faithful (or only retain some properties of the original classes) for various classes, such as graphs of bounded degree, or graphs of bounded pathwidth or treewidth. Section \ref{sec:lbgrid} is devoted to the proofs of Theorem \ref{thm:intro1}
above. We conclude with a discussion and some open problems in Section \ref{sec:ccl}.
\section{Preliminaries}\label{sec:prel}

Given a graph $G$ and a set of vertices $V\subseteq V(G)$, we denote by $G[V]$ the subgraph of $G$ \emph{induced} by $V$, which is the graph with vertex set $V$ such that for every couple of vertices $u,v$ of $V$, $u$ and $v$ are adjacent in $G[V]$ if and only if they are adjacent in $G$.
Similarly, given a set of edges $E\subseteq E(G)$, we denote by $G[E]$ the graph whose vertex set is the set of endpoints of the edges in $E$, and whose edge set is $E$.

\medskip

A \emph{tree-decomposition} of a graph $G$ is a collection $(B_t :t\in V(T))$ of subsets of $V(G)$ (called \emph{bags}) indexed by the nodes of a tree $T$, such that 

\begin{enumerate}
    \item for every edge $uv\in E(G)$, some bag $B_x$ contains both $u$ and $v$, and
    \item for every vertex $v\in V(G)$, the set $\{t\in V(T):v\in B_t\}$ induces a connected subgraph of $T$.
\end{enumerate}
The \emph{width} of $(B_t:t\in V(T))$ is $\max\{|B_t| \colon t\in V(T)\}-1$. 
For $st\in E(T)$, we call $B_s\cap B_t$ an \textit{adhesion}.
The \emph{treewidth} of a graph $G$, denoted by $\mathrm{tw}(G)$, is the minimum width of a tree-decomposition of $G$.  
A \emph{path-decomposition} is a tree-decomposition in which the underlying tree is a path, simply denoted by the corresponding sequence of bags $(B_1,\dots,B_k)$. Similarly, the \emph{pathwidth} of a graph $G$, denoted by $\mathrm{pw}(G)$, is the minimum width of a path-decomposition of $G$. 

For $k\geq 1$, a \emph{$k$-tree} is a graph that can be obtained from a clique of size $k+1$ by iteratively adding vertices whose neighbourhood is a clique of size $k$.
In particular, a $k$-tree admits a tree-decomposition such that each bag has size $k+1$ and induces a clique, and each adhesion has size $k$. 

Analogously to the $k$-trees and the notion of \textit{simple treewidth} (see \cite{HW25} and the references therein), we call a graph $G$ a \textit{simple $k$-path} if it has a path-decomposition $(X_1,\dots,X_m)$ of width $k$ such that 
\begin{itemize}
\item all bags have size $k+1$ (that is $|X_i|=k+1$ for all $i\in [m]$);
    \item all adhesions sets have size $k$ (that is, $|X_i\cap X_{i+1}|=k$ for all $i\in [m-1]$);
    \item all adhesions sets are different ($X_{i}\cap X_{i+1}\neq X_{j}\cap X_{j+1}$ for all distinct $i, j\in [m-1]$);
    \item $G$ is edge-maximal with respect to the path-decomposition $(X_1,\dots,X_m)$ (i.e., $G[X_i]$ induces a clique for each $i\in [m]$).
\end{itemize}
Simple 2-paths are also called \textit{path-like outerplanar graphs} in \cite{BILOSW24}.

\medskip 

We now collect a few useful facts about tree-decompositions that will be needed later. We say a tree-decomposition $(B_t:t\in V(T))$ is \emph{reduced}, if for any $t\neq t'\in V(T)$ it holds that $B_t\not\subseteq B_{t'}$. In particular, in a reduced tree-decomposition, all adhesions sets have size at most $k$ (that is, $|B_s\cap B_{t}|\leq k$ for all $st \in E(T)$).
It can be observed that every $k$-tree has a reduced tree-decomposition in which the bags of every two adjacent nodes differ by exactly one vertex.
\begin{lemma}\label{lem:2-trees}
Suppose that $(B_t:t\in V(T))$ is a reduced tree-decomposition of a graph $G$ of width $k$. Then $|V(T)|\leq \max(|V(G)|-k,1)$. If $G$ is a $k$-tree then equality holds. 
\end{lemma}
\begin{proof}
We proceed by induction on $|V(T)|$. If $T$ has a single vertex, then $G$ contains up to $k+1$ vertices since $(B_t:t\in V(T))$ has width $k$, and thus $|V(T)|=1= \max(|V(G)|-k,1)$. Assume now that $T$ contains at least two vertices. We consider a leaf $\ell$ of $T$ for which the bag $B_{\ell}$ has at most $k$ vertices, if such a leaf exists, otherwise we choose the leaf arbitrarily. Let $y$ be the unique neighbour of $\ell$. The bag $B_{\ell}$ contains a vertex which is not in $B_y$ since the tree-decomposition is reduced.  Then for $T'=T\setminus \ell $ it holds that $(B_t:t\in V(T'))$ is a reduced tree-decomposition of $G'=G[V(G) \setminus (B_{\ell}\setminus B_y)]$. Further, at least one bag of this tree-decomposition has size $k+1$ by the choice of $\ell$, and thus $(B_t:t\in V(T'))$ has width $k$, and $G'$ has at least $k+1$ vertices.  Using the induction hypothesis, $|V(T)|=|V(T')|+1\leq |V(G')|-k+1 \leq |V(G)|-k$. 

The same induction shows that equality holds for $k$-trees, by the observation above stating that $k$-trees have a reduced tree-decomposition in which the bags of every two adjacent nodes differ by exactly one vertex (and thus $|B_{\ell}\setminus B_y|=1$ and $|V(G')|=|V(G)|-1$). 
\end{proof}

We say that a graph $G$ is \emph{$d$-degenerate} if there exists an order $v_1,\ldots,v_n$ on the vertices of $G$ such that for every $1\le i \le n$, $v_i$ has at most $d$ neighbours in $\{v_j : j< i\}$.
In particular, it is well-known that graphs with treewidth $k$ are $k$-degenerate.

\section{Polynomial near-faithful (induced) universal graphs}\label{sec:poly}

We recall the following result mentioned in the introduction.

\begin{theorem}[\cite{chung.graham,GL68}]\label{tree exact} 
    The class of trees has faithful induced-universal graphs of order  $n\mapsto 2^{O(\log^2n)}$, and this order is best possible.
\end{theorem}

However, if we consider only trees of bounded depth, we can significantly decrease the size of the faithful induced-universal graphs.
Indeed, if we denote by $T_{n,k}$ the rooted $n$-ary tree of depth $k$, it is easy to see that for all integers $k$ and $n$, the tree $T_{n,k}$ contains all the rooted trees of depth at most $k$ on at most $n$ vertices. 

The \emph{treedepth} of a graph $G$ is the minimum depth of a rooted forest $F$ such that $G$ is a subgraph of the transitive closure of $F$ (see \cref{sec:td} for more details on this definition and treedepth in general -- we have chosen not to include these in the main body of the paper since these notions are only used once and the proof of \cref{td upper bound} is rather straightforward). The observation above easily implies the following, as we show in \cref{sec:td}.
\begin{corollary}\label{td upper bound}
    For every $k\in \mathbb{N}$, the class of graphs of treedepth at most $k$ has faithful subgraph-universal graphs 
    of order $n\mapsto O(n^{k-1})$.
\end{corollary}

For a graph $G$ and an integer $t$, we denote by $G^{[t]}$ the graph obtained from $G$ by replacing every vertex $u\in V(G)$ by a stable set $S_u$ on $t$ vertices, and every edge $uv\in E(G)$ by a complete bipartite graph $K_{t,t}$ between $S_u$ and $S_v$.

\smallskip

To turn the subgraph-universal graphs of Corollary \ref{td upper bound} into  induced-universal graphs, we will use the following simple lemma.

\begin{lemma}\label{lem:sub to ind}
Let $G$ be a $d$-degenerate graph. Then there exists a subgraph $U$ of $G^{[2^d]}$ such that every subgraph $H$ of $G$ is contained as an induced subgraph in $U$.
\end{lemma}

\begin{proof}
As $G$ is $d$-degenerate, there exists an order $v_1,\ldots,v_n$ on the vertices of $G$ such that for any $1\le i \le n$, the \emph{back-degree} $d^-(v_i):=|\{v_j:j<i, v_iv_j\in E(G)\}|$ of $v_i$ is at most $d$. The vertex set of $U$ is defined as follows: For each vertex $v_i$ of $G$, we replace $v_i$ by a stable set $S_i$ whose size is precisely $2^{d^-(v_i)}$ and we index the vertices of $S_i$ by binary vectors of size $d^-(v_i)$. We now define the edge set of $U$. Consider any vertex $v_i$ and denote by $v_{i_1}, \ldots, v_{i_t}$ the neighbours of $v_i$ preceding $v_i$ in the order, with $t=d^-(v_i)\le d$. For any binary vector $(x_1,\ldots,x_t)$ and any $1\le s \le t$, we add edges between the vertex of $S_i$ indexed by $(x_1,\ldots,x_t)$ and all the vertices of $S_{i_s}$ if and only if $x_s=1$.

Consider any subgraph $H$ of $G$. 
Since $V(H)\subseteq V(G)$, let $v_{i_1}, \ldots,v_{i_\ell}$ be the vertices of $H$, with $i_1<\cdots < i_\ell$. We now explain how to embed $H$ as an induced subgraph in $U$. For each $1\le j\le \ell$, we map $v_{i_j}$ to a vertex of $S_{i_j}$ in $U$ by considering the adjacency in $H$ between $v_{i_j}$ and its neighbours from $V(H)$ in $G$ preceding it in the order. 
By definition of the edge set of $U$, there is always a vertex  $v\in S_{i_j}$ such that 
\[\{k: k<j \text{ and } v \text{ is adjacent to } S_{i_k} \text{ in } U\}=\{k: k<j \text{ and } v_{i_j} \text{ is adjacent to } v_{i_k} \text{ in } H\}.\] Therefore $U$ contains $H$ as an induced subgraph.
\end{proof}

Note that if $G$ has treedepth at most $k$, then $G^{[t]}$ and all its subgraphs have treedepth at most $kt$. Moreover, any graph with treedepth at most $k$ is $k$-degenerate. We thus obtain the following as an immediate consequence of Corollary \ref{td upper bound} and Lemma \ref{lem:sub to ind}.

\begin{corollary}\label{cor: td upper bound iu}
For every $k\in \mathbb{N}$, the class of graphs of treedepth at most $k$ has induced-universal graphs with treedepth $k\cdot 2^{k}$ and order $n\mapsto O(2^{k} \cdot n^{k-1})$.
 \end{corollary}
In fact, we not only give a faithful universal graph for the class of graphs of bounded treedepth: in addition, our universal graph has a decomposition which is universal for all decompositions of graphs of bounded treedepth. 
We now prove a similar result for graphs of bounded pathwidth, except that the universal graphs we obtain are only approximately faithful (they have pathwidth at most a function of $k$, if the original graphs have pathwidth at most $k$). For an integer $k$, let $\mathcal{G}_{\mathrm{pw}\le k}$ denote the class of all graphs of pathwidth at most $k$.

\begin{theorem}\label{pw upper bound}
    For every positive integer $k$, the class $\mathcal{G}_{\mathrm{pw}\le k}$ has subgraph-universal graphs with pathwidth $k^2+k-1$ and  order  $n\mapsto  n^{k+1}$.
\end{theorem}

\begin{proof}
    We construct by induction on $k$
    a graph $U_{k,n}$ which contains every $n$-vertex graph of pathwidth at most $k$ as subgraph.

    First suppose $k=1$. The graphs with pathwidth $1$ are exactly the graphs obtained by adding pendant vertices to a linear forest. Thus, we can take $U_{n,1}$ to be the graph obtained by adding $n-1$ pendant vertices to each vertex of a path on $n$ vertices. This graph has indeed $n^2$ vertices, pathwidth $1$ and contains all the graphs with pathwidth $1$ as subgraph.

    Now suppose $ k\geq 2$ and
    that $(U_{k-1,n})_{n \in \mathbb{N}}$ has been constructed.
    Let $n \in \mathbb{N}$.
    Observe that every $n$-vertex graph of pathwidth at most $k$
    is a subgraph of a connected $n$-vertex graph of pathwidth $k$. Therefore, it is enough to construct subgraph-universal graphs for connected graphs of pathwidth at most $k$.

    If $n \leq k^2+k$, then it is enough to take $U_{k,n} = K_{n}$, which has pathwidth at most $k^2+k-1$.
    Now suppose $n \geq k^2+k+1$.
    Let $S_1, \dots, S_{n-1}$ be $n-1$ pairwise disjoint sets of $k$ new vertices.
    Let $U'_1,\dots U'_{n-2}$ be $n-2$ pairwise disjoint copies of $U_{k-1,n-1}$.
    Finally, let $U_{k,n}$ be the graph constructed from the disjoint union of $U'_1,\dots U'_{n-2}$ and $n-1$ cliques, with respective vertex sets $S_1, \dots, S_{n-1}$, by adding, for each $i\in [n-2]$, all the possible edges between $U'_i$, $S_i$ and $S_{i+1}$ (see \Cref{fig:univ-pw}). More formally, $U_{k,n}$ is defined by 
    \begin{align*}
        V(U_{k,n}) &= \bigcup_{i=1}^{n-1} S_i \cup \bigcup_{i =1} ^ {n-2} V(U'_i) \\
        E(U_{k,n}) &=
            \{ss' : i,j \in [n-1], |i-j| \leq 1, s \in S_i, s' \in S_j\} \\
            &\hspace{7.5mm}\cup \{us : i \in [n-2], u \in V(U'_i), s \in S_i \cup S_{i+1}\} \\
            &\hspace{7.5mm}\cup \{uv : i \in [n-2], uv \in E(U'_i)\}.
    \end{align*}
    \begin{figure}[htb]
        \centering
        \includegraphics[scale=1]{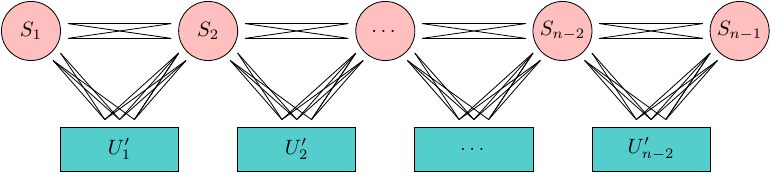}
        \caption{The construction of the subgraph-universal graph $U_{n,k}$ for $\mathcal{G}_{\pw \leq k}$.}
        \label{fig:univ-pw}
    \end{figure}
    
    We now show that $\pw(U_{k,n}) \leq k^2+k-1$.
    For every $i \in [n-1]$,
    let $(B_{i,1}, \dots, B_{i,\ell})$ be a path-decomposition
    of $U'_i$ of width at most $(k-1)^2+(k-1)-1$.
    Then, for every $i \in [n-2]$, for every $j \in [\ell]$,
    let $\beta_{(i-1)\ell + j} = S_{i} \cup S_{i+1} \cup B_{i,j}$.
    The sequence $(\beta_i : i \in [(n-2)\ell])$
    is a path-decomposition of $U_{k,n}$ of width at most
    $2k + (k-1)^2 + (k-1) - 1 = k^2+k-1$.
    This proves that $\pw(U_{k,n}) \leq k^2+k-1$.

    Consider a connected $n$-vertex graph $G$
    admitting a path-decomposition $(B_i : i \in [q])$
    of width at most $k$.
    We will show that $G$ is a subgraph of $U_{k,n}$.

    By taking a reduced path-decomposition,  we may suppose $1 \leq |B_i \cap B_{i+1}| \leq k$, and $B_{i+1} \setminus B_{i} \neq \emptyset$, for every $i \in [q-1]$ and by \cref{lem:2-trees}, $q\leq n-k\leq n-1$.
    Assume first that $q=1$, and consider some vertex $u \in V(G)$.
    Observe that $\pw(G-u) \leq \pw(G)-1 \leq k-1$,
    and so $G-u$ is a subgraph of $U'_1$ by induction.
    Therefore, extending this embedding by mapping $u$ to a vertex in $S_1$,
    we deduce that $G$ is a subgraph of $U_{k,n}$.
    We can now assume that $2 \leq q \leq n-1$.

    Let $\{A_1, \dots, A_r\}$ be a maximal
    family of pairwise disjoint members of $\{B_i \cap B_{i+1} : i \in [q-1]\}$.
    Suppose $A_1, \dots, A_r$ appear in this order in the
    path-decomposition.
    For every $i \in [q-1]$, by maximality of the family $A_1, \dots, A_r$, and because $B_i \cap B_{i+1}$ is nonempty, $B_i \cap B_{i+1}$ intersects $\bigcup_{j=1}^r A_j$.
    Therefore, for every $i \in [q]$, $|B_i \setminus \bigcup_{j=1}^r A_j| \leq |B_i|-1 \leq k$. 
    We deduce that the path-decomposition
    $(B_i \setminus \bigcup_{j=1}^r A_j : i \in [q])$
    of $G - \bigcup_{j=1}^r A_j$ has width at most $k-1$,
    and so $\pw(G - \bigcup_{j=1}^r A_j) \leq k-1$.
    If $r=1$, then 
    map the vertices in $A_1$ to vertices in $S_1$,
    and vertices in $G-A_1$ as a subgraph of $U'_1$.
    Now suppose $r \geq 2$.
    Since the $(A_i : i \in [r])$ are separators in a path-decomposition,
    for every component $C$ of $G - \bigcup_{i=1}^r A_i$,
    there exists $i\in [r-1]$
    such that $N_G(V(C)) \subseteq A_i \cup A_{i+1}$.
    As a consequence,
    there is a partition $V_1, \dots, V_{r-1}$ of $V(G) \setminus \bigcup_{i=1}^r A_i$ with some potentially empty parts,
    such that, for every $i \in [r-1]$,
    $N_G(V_i) \subseteq A_i \cup A_{i+1}$ (note that each part $V_i$ is does not necessarily induce a connected subgraph in $G$).
    For every $i \in [r-1]$,
    since $\pw(G[V_i]) \leq \pw(G - \bigcup_{i=1}^r A_i) \leq k-1$, 
    $G[V_i]$ can be embedded as a subgraph of $U'_i$.
    Combining these embeddings,
    and mapping vertices in $A_i$ to vertices in $S_i$
    for every $i \in [r]$,
    we conclude that $G$ is a subgraph of $U_{k,n}$.
\end{proof}

Note that if $G$ has pathwidth (resp.\ treewidth) at most $k$, then $G^{[t]}$ and all its subgraphs have pathwidth (resp.\ treewidth) at most $kt$. Moreover, any graph with treewidth or pathwidth at most $k$ is $k$-degenerate. We thus obtain the following as an immediate consequence of Theorem \ref{pw upper bound} and Lemma \ref{lem:sub to ind}.

\begin{corollary}\label{cor: pw upper bound}
For every $k\in \mathbb{N}$, the class $\mathcal{G}_{\mathrm{pw}\le k}$ has induced-universal graphs with pathwidth $(k^2+k-1)2^{k^2+k-1}=2^{O(k^2)}$ and order $n\mapsto 2^{k^2+k-1} n^{k+1}$.
\end{corollary}

We remark that the bounds on the order of the universal graphs in Theorem \ref{pw upper bound} and Corollary \ref{cor: pw upper bound} can be improved by a $\Theta(\log n)$ factor by starting the induction with $|U_{n,1}|=O(n^2/\log n)$ using \cite{CGS81}. As the improvement is minor, we omit the details. 

\medskip

We now focus on graphs of bounded treewidth.
We will need the following lemma.
\begin{lemma}[\cite{Bodlaender1998,ElberfeldJakobyTantau10}]
\label{lem:log_depth}
    There is a constant $c>0$ such that every $n$-vertex graph with treewidth $k$ admits a tree-decomposition of width at most $3k-1$ whose underlying tree is binary and has depth at most $c\log n$.
\end{lemma}

It turns out that it is easy to provide ``universal tree-decompositions'' for tree-decompositions of  diameter $O(\log n)$. When applying the lemma above, we unfortunately do increase the width of such a decomposition, but we note that already for trees some increase in width is necessary.  For an integer $k$, let $\mathcal{G}_{\mathrm{tw}\le k}$ denote the class of all graphs of treewidth at most $k$.
\begin{theorem}\label{tw upper bound}
    For every $k\geq 1$, the class $\mathcal{G}_{\mathrm{tw}\le k}$ has induced-universal graphs with treewidth $3k-1$ and order $n\mapsto n^{O_k(1)}$. 
\end{theorem}
\begin{proof}
    We first construct a graph $U_{n,k}$ that contains all $n$-vertex graphs in $\mathcal{G}_{\mathrm{tw}\le k}$ as \textit{subgraphs}, and then describe how to adjust it so it contains all such graphs as \textit{induced subgraphs}. (Note that we could also use Lemma \ref{lem:sub to ind} for this but doing so would make the treewidth of the induced-universal graph exponential in $k$, whereas in our proof the treewidth remains at most $3k-1$.)
    
    We describe a rooted tree-decomposition $(B_t:t\in V(T_U))$  of  $U_{n,k}$  with bags of size $\ell=3k$ and define $U_{n,k}$ such that every bag induces a clique in $U_{n,k}$. The tree $T_U$ of the decomposition is a $2^{\ell+1}$-ary tree of depth $d = \lceil c\log n \rceil$
    (where $c$ is the constant from Lemma \ref{lem:log_depth}) and the bags are recursively created as follows:
    \begin{itemize}
        \item In iteration $0$ we create a root bag $R$ consisting of $\ell$ vertices. 
        \item For $i\in \{1,\dots, d\}$, for each bag $B$ created in iteration $i-1$, for each subset $S\subseteq B$, we create two child bags of $B$, each consisting of $S$ together with $\ell-|S|$ new vertices. 
    \end{itemize}      
    The number of vertices in $U_{n,k}$ is upper bounded by $\ell\cdot |V(T_U)| \leq \ell \sum_{i=0}^d (2^{\ell+1})^i \leq 2\ell \cdot 2^{(\ell+1)d}=6k\cdot 2^{(3k+1)d}$.
    As $d = O(\log n)$, 
    $|V(U_{n,k})|=n^{O_k(1)}$. 
    
    It remains to prove that any $n$-vertex graph of treewidth at most $k$ can be embedded as a subgraph of $U_{n,k}$. Let $G$ be a graph of treewidth at most $k$ and let $(B'_x:x\in V(T_G))$ be a rooted tree-decomposition of width $\leq 3k$ with a binary tree of depth at most $ c \log n$; the existence of such a tree-decomposition follows from \cref{lem:log_depth}. We moreover assume that there are no repeated bags. We show that the tree-decomposition of $G$ can be embedded in the tree-decomposition of $U_{n,k}$ which proves the desired result.
    
    We first give an informal argument that we formalise afterwards.  The main idea is to map bags $B'_{t'}$ of level $i$ in $T_G$ to bags $B_{t}$ of level $i$ in $T_U$. For $i=0$ we embed the root bag of $T_G$ into the root bag of $T_U$ (by which we mean we create an injection between the bags and the elements in the bags). Inductively, for $i\geq 1$ we then greedily embed each bag $B'_{t'}$ of $T_G$ at level $i$ into a bag $B_{t}$ of $T_U$ at level $i$ which is the child bag of the bag $B_{p}$ of $T_U$ at level $i-1$ that we embedded $B'_{p'}$ into, for $p'$ the parent of $t'$ in $T_G$.

   To make this informal argument precise, we describe an injective mapping $f\colon V(G) \to V(U_{n,k})$ and an injective mapping $g\colon V(T_G) \to V(T_u)$ which maps 
    each node $x\in V(T_G)$ to a node $y =g(x)\in V(T_U)$ such that $f(B'_x)=B_y\cap f(V(G))$.     

     We proceed from the root of $T_G$ to the leaves. First, $g$ maps the root $r'$ of $T_G$ to the root  $r$ of $T_U$, and $f$ maps injectively the vertices of the root bag $B_{r'}'$ to vertices of the root bag $B_{r}$. Suppose we have already described the image of a node $x\in V(T_G)$ by $g$ and the images of the vertices of $B_x'$ by $f$, but not the images of the children $x_1$ and $x_2$ of $x$ under $g$, and the images of the vertices of  $B'_{x_1}\setminus B_x'$ and $B'_{x_2}\setminus B_x'$ under $f$. 
     Then $g$ maps  $x_1$ to a child of $g(x)$
     such that $f(B'_{x_1}\cap B'_{x})=B_{g(x_1)}\cap f(B_{g(x)})$, 
     and $f$ maps the vertices of $B'_{x_1}\setminus B'_{x}$ injectively to the vertices of $B_{g(x_1)}\setminus f(B_{g(x)})$. 
     We do the same for $x_2$ if $x_1\ne x_2$ by defining $g(x_2)$ as a child of $g(x)$ distinct from $g(x_1)$ such that $f(B'_{x_2}\cap B'_{x})=B_{g(x_2)}\cap f(B_{g(x)})$.  
     Note that $g$ is such that $\{g(B'_x):x \in T_G\}$ induces a subtree of $T_U$ and since $T_G$ has height at most $c \log n$, its image has the same height, which means this map is well-defined. This shows that the tree-decomposition of $G$ can be embedded in the tree-decomposition of $U_{n,k}$, as desired.
    
     We now explain how to obtain a construction which works for induced subgraphs. As the proof is very similar as the proof above, we only explain the main new ingredients and omit the technical details. We note that there are only $2^{O(\ell^2)}=O_k(1)$ possibilities for which graphs may be induced in a bag. More precisely, the construction for $T_U$ and $U$ is as follows:
    \begin{itemize}
    \item In iteration $-1$, we create an empty root bag $B_{\emptyset}$.
        \item In iteration $0$, for each graph $H$ on $\ell$ vertices, we create a bag $B_H$ consisting of $\ell$ vertices which induce the graph $H$ in $U$, and we join all bags $B_H$ to the root bag $B_{\emptyset}$ in $T_U$.
        \item For $i\in \{1,\dots, d\}$, for each bag $B$ created in iteration $i-1$, for each subset $S\subseteq B$, for each graph $H$ on $\ell$ vertices containing $U[S]$ as induced subgraph, we create two child bags $B'$ and $B''$ of $B$ consisting of $S$ and $\ell-|S|$ new vertices. In the graph $U$, we already defined which edges are present for vertices in $S$. We place edges between vertices in $B'$ and $B''$ to ensure $U$ induces $H$ on the bag (while keeping the edges between vertices of $S$ exactly the same).
    \end{itemize}  
    Since we still create $O_k(1)$ new bags for each bag of the previous iteration, and repeat this $O(\log n)$ times, the resulting graph has $n^{O_k(1)}$ vertices, as desired.
\end{proof}
We obtain the following immediate corollaries of Theorem \ref{tw upper bound} for $k=1$ and $k=2$, where we use the fact that a graph has treewidth at most 2 if and only if it is $K_4$-minor-free, that $K_4$-minor-free graphs are planar, and the fact that $K_7$ has treewidth 6.

\begin{corollary}\label{cor:poly upper bound1}
The class of trees has induced-universal graphs of treewidth 2 and polynomial order.  In particular the class of trees has planar induced-universal graphs of polynomial order.
\end{corollary}
\begin{corollary}\label{cor:poly upper bound2}
The class of graphs of treewidth 2 has induced-universal graphs of treewidth 5 and polynomial order.  In particular the class of $K_4$-minor-free graphs has $K_7$-minor-free induced-universal graphs of polynomial order.
\end{corollary}
The result of Corollary~\ref{cor:poly upper bound1} provides a positive answer to Problem~\ref{pro:3}.

\medskip

In \cite{BILOSW24}, it is shown that the class of outerplanar graphs has planar subgraph-universal graphs of order $n\mapsto 2^{O(\log^2n)}$. In the proof of their result, the authors show the following result.
\begin{lemma}[Lemma 13 in~\cite{BILOSW24} ]
\label{lem:treewidth3forpw2}    
For every $n\geq 3$, there is a graph $H_n$ of treewidth at most 3 on less than $n^2$ vertices with an edge $e^*\in E(H_n)$ such that for each $n$-vertex simple 2-path $G$ and each edge $e=uv$ of $G$ with $d(u)=2$, $G$ can be embedded as subgraph of $H_n$ in such a way that $e$ is mapped to $e^*$.
\end{lemma}

In~\cite{BILOSW24}, simple 2-paths are called \textit{path-like outerplanar graphs}. Moreover, it is not explicitly mentioned in their lemma that the graph has treewidth 3 but this follows directly from the construction (see more specifically Figure 9 in~\cite{BILOSW24}).

\begin{theorem}\label{thm:K5-m-free universal for K4-m-free}
The  class of graphs of treewidth at most 2 has subgraph-universal graphs of treewidth at most 3  and order at most $n\mapsto n^{20\log_2 n}$.
\end{theorem}

\begin{proof}
We give an inductive construction for a graph $U_n$ of treewidth at most 3 which contains every $n$-vertex  graph $G$ of treewidth at most 2 as a subgraph.  
We will construct $U_n$ with a special edge $e_n$ of $U_n$ such that for any $e'\in E(G)$,
we can embed $G$ into $U_n$ such that $e'$ is mapped to $e_n$. This will help us to embed our graphs recursively.

Let $U_{3}$ be a triangle with an arbitrary edge assigned as $e_3$. Then indeed $|V(U_3)|=3\leq 3^{20\log_2 3}$.

Let $n\geq 4$. We assume that $U_{\lceil n/2+1\rceil}$ has already been constructed ($\lceil 4/2+1\rceil=3$) with the desired properties.
Let $H_n$ be the graph from Lemma~\ref{lem:treewidth3forpw2} on less than $n^2$ vertices and treewidth at most 3, with edge $e^*\in E(H_n)$ such that for each $n$-vertex simple 2-path $G$ and each edge $e=uv$ of $G$ with $d(u)=2$, $G$ can be embedded as subgraph of $H_n$ in such a way that $e$ is mapped to $e^*$.

We set $e_n=e^*$. We obtain $U_n$ by gluing to every edge $e$ of $H_n$ a set of $n-1$ copies of $U_{\lceil n/2 +1 \rceil} $ such that $e$ is identified with every edge $e_{\lceil n/2+1 \rceil} $ in $U_{\lceil n/2 +1 \rceil} $. Since $U_n$ is obtained by repeatedly gluing graphs of treewidth at most $3$ on edges (i.e., cliques of size 2), $U_n$ has treewidth at most $3$.

We now bound the number of vertices in $U_n$. Note that since $H_n$ has treewidth $3$, $H_n$ is $3$-degenerate and thus $|E(H_n)|\leq 3|V(H_n)|$, so
\begin{align*}
    |V(U_n)|=|V(H_n)|+(n-1)|E(H_n)|\cdot (|V(U_{\lceil n/2+1 \rceil})|-2)\leq 3 n |V(H_n)|\cdot |V(U_{\lceil n/2+1 \rceil})|.
\end{align*}
For $n\geq 4$ we have $\lceil n/2+1\rceil\leq  \tfrac78 n$, and thus, since $|V(H_n)|\leq n^2$ and $|V(U_{\lceil n/2+1 \rceil})| \leq n^{20 \log_2(\lceil n/2+1 \rceil)}$, we find that
\begin{align*}
    |V(U_n)|\leq 3n^3n^{20 \log_2(\lceil n/2+1 \rceil)}\leq 3n^{3+20(\log_2 n+\log_2(7/8))}\leq n^{20\log_2 n},
\end{align*}
where we have used the fact that $n^{3+20\log_2(7/8)}\le 1/3$ for any $n\ge 4$.

Since we only need to embed the graphs as subgraphs,  we can restrict ourselves to 2-trees. Let $G$ be a 2-tree. So $G$ has a tree-decomposition $(B_t: t\in V(T))$ in which each bag induces a clique of size 3 and each adhesion has size $2$. Let $e'=uv$ be an edge of $G$. We show that we can embed $G$ into $U_n$ as a subgraph, while mapping $e'=uv$ to $e_n$.

Consider the vertex set $\{u,v\}$. Note that there is at most one component of $G\setminus \{u,v\}$ with more than $\frac{n-2}{2}$ vertices. Choose the largest component $D'$ of $G\setminus \{u,v\}$ and set $G'=G[D'\cup \{u,v\}]$. We now want to define $t'$ such that $\{u,v\}\subseteq B_{t'} \subseteq G'$. If $\{u,v\}$ is not a cut-set, then $G'=G$ and it suffices to choose any $t'$ such that $\{u,v\}\subseteq B_{t'}$, which exists since $uv$ is an edge.
If $\{u,v\}$ is a cut-set, we consider an arbitrary bag in $(B_t: t\in V(T))$ containing a vertex which is not in $D'\cup \{u,v\}$ and the closest bag to it, denoted $B_{t'}$, containing a vertex in $D'$. Then $B_{t'}$ contains exactly one vertex from $D'$, and since every bag has size $3$ and induces a clique, the two other vertices in $B_{t'}$ can only be $u$ and $v$ as no other vertex not in $D'$ has neighbours in $D'$ and all bags have size 3.

We construct a path $P$ in $T$ starting at $t'$ such that each component of $G\setminus ( \bigcup_{p\in V(P)}B_p)$ has size at most $\frac{n-2}{2}$. Let $t_1=t'$ and suppose we defined $t_1, \ldots, t_i$ for some $i\geq 1$. There is at most one neighbour $s$ of $t_i$ for which the subtree $T'$ of $T\setminus t_i$ containing $s$ satisfies $|(\bigcup_{t\in V(T')}B_t)\setminus B_s|> \frac{n-2}{2}$. If there is no such neighbour or this neighbour is $t_{i-1}$, then we abort. Otherwise, let $t_{i+1}$ be such a neighbour and we continue. We end up with a path $P=\{t_1,\dots,t_\ell\}$ starting at $t'$ such that all components of  $G\setminus ( \bigcup_{p\in V(P)}B_p)$ have size at most $\frac{n-2}{2}$.

Then the graph $G_0 =G[\bigcup_{p\in V(P)}B_p]$ is a simple 2-path, since the tree-decomposition restricted to the path $P$ gives a path-decomposition for $G_0$. Moreover, by construction, $G_0$ is an induced subgraph of $G'$, since every connected component of $G\setminus \{u,v\}$ except $D'$ is a connected component of $G\setminus B_{t'}$ and has size at most $\frac{n-2}{2}$. 
Thus $u$ or $v$ has degree $2$ in $G_0$ as the first adhesion (i.e. the intersection of the first two bags) of the $2$-path cannot equal $\{u,v\}$, because then $\{u,v\}$ would be a cut-set of $G'$.

We next show how to attach the components $D$ of $G \setminus (\bigcup_{p \in V(P)}V(B_p))$. 
Let $D$ be such a component. Then $|D|\leq \frac{n-2}{2} \leq \lceil n/2-1\rceil$ by construction of $P$. 

There are exactly two vertices $u',v'\in (\bigcup_{p \in V(P)}V(B_p))\cap N_G(D)$ and since they share a bag, $u'v'\in E(G)$. Let $f$ be the edge of $U_n$ that $u'v'$ got mapped to when embedding $G_0$ into the copy of $H_n$ in $U_n$. 
Since $|D\cup \{u',v'\}|\leq  \lceil n/2+1\rceil$, we can embed $G[D\cup \{u_f,v_f\}]$ in a copy of $U_{\lceil n/2+1 \rceil}$ containing $f$, while sending $u'v'$ to $f$. 
Repeating this for each component (always using a different copy of $U_{\lceil n/2+1 \rceil}$), this embeds $G$ into $U_n$.
\end{proof}

As graphs of treewidth 3 are $K_5$-minor-free, we obtain the following as an immediate corollary.

\begin{corollary}\label{cor:K5-m-free universal for K4-m-free}
The  class of $K_4$-minor-free graphs has $K_5$-minor-free subgraph-universal graphs  of order $n\mapsto n^{20\log_2 n}$.
\end{corollary}

\section{Exponential lower bounds for faithful universal graphs}\label{sec:lb}

In this section we obtain exponential lower bounds on the size of faithful universal graphs for various families, including bounded degree graphs, graphs of bounded treewidth or pathwidth, and $K_t$-minor-free graphs.

\subsection{Bounded degree}

Let $\mathcal{F}_\Delta$ denote the class of graphs with maximum degree at most $\Delta$.
Recall that there are induced-universal graphs of order $O(n^{\Delta/2})$ for $\mathcal{F}_\Delta$ \cite{AN19}, and that this bound is best possible. We now show that if we fix the maximum degree $d$ of our universal graphs to be  $n^{o(1)}$, then such universal graphs need to have order $n^{(\Delta/2-1-o(1)) n}$, which is (asymptotically) not much better than taking the disjoint union of all $n$-vertex graphs of maximum degree $\Delta$. 

\begin{theorem}\label{thm:degree}
Let $U$ and $U'$ be two graphs with maximum degree $d$. If $U$ contains all $n$-vertex graphs from $\mathcal{F}_\Delta$ as induced subgraphs and $U'$ contains all $n$-vertex graphs from $\mathcal{F}_\Delta$ as subgraphs, then \[|V(U)|\ge \exp\left(n\big((\Delta/2-1)\log n-O_\Delta(1)-\log(4d)\big)\right), \text{ and} \]
\[|V(U')|\ge \exp\left(n\big((\Delta/2-1)\log n-O_\Delta(1)-\log(4d)-2\Delta  \log \tfrac{de}{\Delta}\big)\right).\]
In particular, if $d=n^{o(1)}$, then $|V(U)|\ge n^{(\Delta/2-1-o(1)) n}$ and $|V(U')|\ge n^{(\Delta/2-1-o(1)) n}$.
\end{theorem}

\begin{proof}
Let $\mathcal{G}_\Delta$ be the family of all (unlabelled) $\Delta$-regular connected graphs on $n$ vertices, and let $g_\Delta(n)$ denote the cardinality of this family (i.e., number of non-isomorphic $\Delta$-regular connected graphs on $n$ vertices). It is known that $\log g_\Delta(n)=(\Delta/2-1)n\log n-O_\Delta(n)$ \cite{Bol80,Bol82}, where the implicit multiplicative constant in the second term on the right-hand side only depends on $\Delta$. 
As each graph from $\mathcal{G}_\Delta$ appears as an induced subgraph in $U$,  there is a vertex $v^*\in V(U)$ which is contained in the induced copies of at least a $1/|V(U)|$ fraction of these graphs in $V(U)$.

We claim that there are at most $(4d)^n$ subsets $S\subseteq V(U)$ of size $n$ containing $v^*$ and inducing a connected subgraph of $V(U)$. 
Indeed, to describe any such subset $S$ we can consider an arbitrary (but fixed) ordering  of the neighbours of each vertex $u$ in $U$ as $u[1],u[2],\ldots,u[d_U(u)]$. We then
choose any spanning tree $T$ of $U[S]$ and root it at $v^*$, and label each edge $xy$ of $T$, where $x$ is the parent of $y$,  with the integer $i\in \{1,\ldots,d\}$ such that
$y=x[i]$. Note that the description of $T$ as an unlabelled rooted tree together with the edge-labelling is enough to inductively reconstruct the set $S$, starting from the root $v^*$. This can be done inductively as follows: the labels of the edges incident to the root of $T$ allow us to recover the neighbours of $v^*$ included in $S$ in the graph $U$. For each child $t$ of the root of $T$, we then consider the subtree of $T$ rooted at $t$ and by induction we can recover the remainder of the vertices of $S$. This is illustrated in Figure \ref{fig:reconstruction}.

\begin{figure}[htb]
    \centering
    \includegraphics[scale=1]{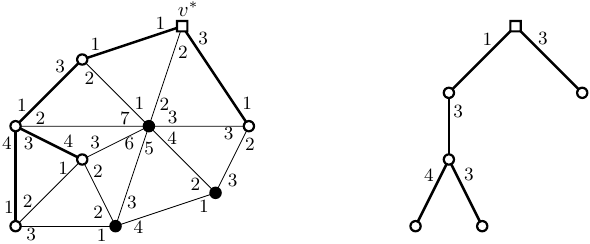}
    \caption{Illustration of the reconstruction of the set $S$ (left, depicted with white nodes) given the rooted tree $T$ and its edge-labelling (right). For the sake of readability the ordering on the neighbours $v$ of each vertex $u\in V(U)$ is written as a label on the edge $uv$, close to $u$}
    \label{fig:reconstruction}
\end{figure}

Unlabelled \emph{ordered} rooted trees are counted by Catalan numbers (see for instance \cite{Stan15}), so there are at most  $\tfrac1{2n-1}{\binom{2n-1}{n-1}}\le 4^n$ unlabelled  rooted trees on $n$ vertices. As there are at most $d^n$ possible edge-labellings for each tree,  there are at most $4^n\cdot d^n=(4d)^n$ 
subsets $S\subseteq V(U)$ of size $n$ containing $v^*$ and inducing a connected subgraph of $V(U)$. 

It follows that $(4d)^n\ge g_\Delta(n)/|V(U)|$, and thus \[|V(U)|\ge g_\Delta(n)/(4d)^n=\exp\left(n\big((\Delta/2-1)\log n-O_\Delta(1)-\log(4d)\big)\right),\] as desired.

\medskip

We now consider $U'$. As each graph from $\mathcal{G}_\Delta$ appears as a subgraph in $U'$,  there is a vertex $v^*\in V(U')$ which is contained in the  copies of at least a $1/|V(U')|$ fraction of these graphs in $V(U')$. To describe such a copy $H$ in $U'$, it suffices to describe its vertex set $S\subseteq V(U')$ (as in the paragraph above, there are at most $(4d)^n$ such $n$-vertex subsets containing $v^*$), and the subset $F\subseteq E(U'[S])$ of edges of $U'[S]$ corresponding to $H$. There are  $\binom{|E(U'[S])|}{|F|}$ choices for $F$. 
As $H$ has maximum degree $\Delta$ and $U'$ has maximum degree $d$, $|F|\le \Delta n$ and $|E(U'[S])|\le d n$, and in particular there are at most $2^{dn}$ choices for $F$. If $d\le 2\Delta$, then $2^{dn}\le 2^{2\Delta n}\le \exp(2\Delta n)$. If $d\ge 2\Delta$, then 
there are \[\binom{|E(U'[S])|}{|F|}\leq\binom{dn}{\Delta n}\le \left(\frac{de}{\Delta }\right)^{\Delta n}= \exp(\Delta n \log \tfrac{de}{\Delta})\le \exp(2\Delta n \log \tfrac{de}{\Delta})\] choices for $F$, for any fixed set $S$ of size $n$. 

\smallskip

It follows that $(4d)^n \exp(2\Delta n \log \tfrac{de}{\Delta}) \ge g_\Delta(n)/|V(U')|$, and thus \[|V(U')|\ge 
\exp\left(n\big((\Delta/2-1)\log n-O_\Delta(1)-\log(4d)-2\Delta  \log \tfrac{de}{\Delta}\big)\right),\] as desired.
\end{proof}

Theorem \ref{thm:degree} implies in particular that faithful induced-universal and subgraph-universal graphs for $\mathcal{F}_\Delta$ have order at least $n\mapsto n^{(\Delta/2-1-o(1)) n}$. 

\subsection{Treewidth, pathwidth, and minors}

Recall from the preliminaries that simple $k$-paths admit a path-decomposition $(X_1,\dots,X_m)$ such that all bags induce cliques of size $k+1$ and all adhesions are different and of size $k$. This path-decomposition is reduced thus by \cref{lem:2-trees}, $m = \max(n-k,1)$.
It can be checked that if  $G$ is a simple $k$-path with at least $k+2$ vertices, then $G$ contains exactly two vertices of degree at most $k$, and these two vertices have degree exactly $k$.

\begin{lemma}
\label{lem:num_simple_pwk_graphs}
Let $k\geq 2$ and $n\geq k+4$ be integers.
Let $\mathcal{S}_{n,k}$ be the family of unlabelled $n$-vertex simple $k$-paths (that is, we consider graphs up to isomorphism). Then    $|\mathcal{S}_{n,k}|\geq k^{n-2k-2}$.
\end{lemma}
\begin{proof}
Let $k\geq 2$ and $n\geq k+4$ be given. For $i\in [k+1]$, we set
\[
B_i=\{i,i+1,\dots,i+k\}.
\]
Next, we create a sequence of bags $B_{k+2},\dots,B_{n-k}$ via the following procedure. 
Having defined $B_i=\{y_1,y_2,\dots,y_{k},i+k\}$ with $y_1<y_2<\dots<y_k<i+k$ for some $i\in \{k+1,k+2,\dots,n-k-1\}$, we define $B_{i+1}$ as follows. We remove $y_j$ for some $j\in [k]$ and add the element $i+k+1$. We record $x_{i-k}=j$. 

Each possible sequence of record $x=(x_{i})_{i \in [n-2k-1]}\in [k]^{n-2k-1}$ is associated with a distinct sequence of bags containing labelled vertices. 
From such a sequence of bags, we define the graph $G(x)$ as the labelled graph on vertex set $[n]$ where two vertices are adjacent if they appear in  a common bag. 
By construction $G(x)$ is a simple $k$-path.
Two examples are given in Figure \ref{fig:outerplanar-pathlike} below for $k=2$.
\begin{figure}[htb]
    \centering
    \begin{tikzpicture}[node distance={10mm}, thick, main/.style = {draw, circle}]

\node[main] (1) {1};
\node[main] (2) [below right of=1] {2};
\node[main] (3) [above right of=2] {3};
\node[main] (4) [below right of=3] {4};
\node[main] (5) [above right of=4] {5};
\node[main] (6) [below right of=5] {6};
\node[main] (7) [above right of=6] {7};

\draw (1) -- (2);
\draw (1) -- (3);
\draw (3) -- (4);
\draw (3) -- (5);
\draw (5) -- (6);
\draw (5) -- (7);
\draw (4) -- (6);
\draw (2) -- (4);
\draw (6) -- (7);
\draw (2) -- (3);
\draw (4) -- (5);

\node[below=10mm of 4] {\(G((1,1))\)};

\node[main] (1b) [right=70mm of 1] {1};
\node[main] (2b) [below right of=1b] {2};
\node[main] (3b) [above right of=2b] {3};
\node[main] (4b) [below right of=3b] {4};
\node[main] (5b) [above right of=4b] {5};
\node[main] (6b) [below right of=5b] {6};
\node[main] (7b) [below right of=4b] {7};

\draw (1b) -- (2b);
\draw (1b) -- (3b);
\draw (3b) -- (4b);
\draw (3b) -- (5b);
\draw (5b) -- (6b);
\draw (4b) -- (7b);
\draw (4b) -- (6b);
\draw (2b) -- (4b);
\draw (6b) -- (7b);
\draw (2b) -- (3b);
\draw (4b) -- (5b);

\node[below=10mm of 4b] {\(G((1,2))\)};

\end{tikzpicture}
    \caption{Examples of two simple 2-paths $G(x)$ for $x=(1,1)$ and $x=(1,2)$.}
    \label{fig:outerplanar-pathlike}
\end{figure}

\medskip

We count the number of isomorphism types we created. It is immediate that different $x$ create different labelled graphs. Suppose that $x$ and $x'$ create isomorphic graphs $G(x)$ and $G(x')$ and let $f\colon V(G(x))\to V(G(x'))$ be a graph isomorphism. 
We have the following two properties:
\begin{itemize}
    \item $f(1)\in \{1,n\}$. Indeed, a graph isomorphism must preserve the degrees, and $1,n$ are the only vertices of degree $k$ in both graphs.
    \item $f(1)=1$ implies that $x=x'$. This can be shown inductively. We find that $f(2)$ is $2$, since it is the unique neighbour of $1$ of degree $k+1$. Similarly, $f(3)=3,\dots,f(k+1)=k+1$ since they are adjacent to $1$ and all are identified by their degree.

    Suppose that for some $i\geq k+1$, we have shown that $f(j)=j$ for all $j\in \{1,\dots,i\}$. The vertex $i+1$ is the unique vertex greater or equal to $i+1$ which is adjacent to $k$ vertices in $\{1,\dots,i\}$, in both $G(x)$ and $G(x')$ (we use the third property of simple $k$-paths). Hence $f(i+1)=i+1$. 
    So we find by induction that $f$ is the identity function and it follows also that the records $x=x'$ are the same.
\end{itemize}

Suppose we have an isomorphism type containing three distinct records $\{x,x',x''\}$, then the isomorphism $g:G(x)\to G(x'')$ must have $g(1)=n$ (otherwise $x=x''$) and $f:G(x)\to G(x')$ must have $f(1)=n$. This implies that the isomorphism $g\circ f^{-1}$ from $G(x')$ to $G(x'')$ maps $n$ to $n$ and so must map $1$ to $1$, which implies that $x'=x''$, a contradiction.

This means that if we partition the $x\in [k]^{n-2k-1}$ by the isomorphism type of $G(x)$, then all parts have size at most $2$. This implies that the number of non-isomorphic simple $k$-paths we created is at least $\frac12 k^{n-2k-1} \geq k^{n-2k-2}$. 
\end{proof}

We now show that any graph of treewidth at most $k$ which contains many non-isomorphic simple $k$-paths must be large by generalizing the argument of Theorem 11 in~\cite{BILOSW24}.

\begin{lemma}
\label{lem:pw2_lb}
Let $k\geq 2$ and $n\geq k+4$.
Let $\mathcal{S}_{n,k}$ be the family of all $n$-vertex simple $k$-paths.
If $U$ is a graph of treewidth $k$ that contains all graphs from $\mathcal{S}$ as subgraph for some $\mathcal{S}\subseteq \mathcal{S}_{n,k}$, then 
\[
|V(U)|^2\geq  |\mathcal{S}|.
\]
\end{lemma}
\begin{proof}
Suppose that $U$ is a graph of treewidth $k$ containing all graphs in $\mathcal{S}$ as subgraph.
We consider that $V(G) \subseteq V(U)$ and $G \subseteq U[V(G)]$
for every $G \in \mathcal{S}$.
Let $(B_t :t\in V(T))$ be a tree-decomposition for $U$ with all bags of size at most $k+1$. We may assume the tree-decomposition is reduced, and thus that no bag is a subset of any other bag. 
It follows from \cref{lem:2-trees} that there are at most $|V(U)|-k$ bags.
Let $G\in \mathcal{S}$ and let $u,v$ be the unique vertices of degree $k$ in $G$.
Let $(W_1,\dots,W_m)$ be a reduced path-decomposition for $G$ of width at most $k$. Note that
$G[W_i]$ is a clique for every $i \in [m]$,
$u\in W_1$,
and $v\in W_m$. 
Since $G[W_1]$ and $G[W_m]$ both induce a $(k+1)$-clique, there exist $t_{G,1},t_{G,2}\in V(T)$ such that $W_1 \subseteq B_{t_{G,1}}$ and $W_m \subseteq B_{t_{G,2}}$. By size constraints, it follows that $W_1 = B_{t_{G,1}}$ and $W_m = B_{t_{G,2}}$.

Let $G'\in \mathcal{S}$ and let $u',v'$ be its unique vertices of degree $k$ in $G'$. We define $t_{G',1},t_{G',2}\in V(T)$ in the same manner as for $G,u,v$. We claim that if $\{t_{G,1},t_{G,2}\}=\{t_{G',1},t_{G',2}\}$, then $G\cong G'$. Since $|V(T)|\leq |V(U)|-k$, this claim implies that
\[
(|V(U)|-k)^2\geq |\mathcal{S}|,
\]
and thus $|V(U)|^2\ge |\mathcal{S}|$. 

Suppose towards a contradiction that $\{t_{G,1},t_{G,2}\}=\{t_{G',1},t_{G',2}\}$ yet $G\not\cong G'$.
Let $t_{G,1}=t_1,t_2,\dots,t_\ell=t_{G,2}$ be the path in $T$ from $t_{G,1}$ to $t_{G,2}$. 
We claim that 
\begin{equation}
    \label{eq:AG}
    V(G)=B_{t_1}\cup B_{t_\ell}\cup \left(\bigcup_{i=1}^{\ell-1} (B_{t_i}\cap B_{t_{i+1}})\right).
\end{equation}
Firstly, $B_{t_1}\cup B_{t_\ell}\subseteq V(G)$ by choice of $t_{G,1},t_{G,2}$.
So for every $i \in [\ell-1]$,
the set $V(G) \cap B_{t_i}\cap B_{t_{i+1}}$ separates $W_1$ and $W_\ell$ in $G$, and so must have size at least $k$.
Since there are no repeated bags by assumption, $|B_{t_i}\cap B_{t_{i+1}}|=k$ and so the inclusion ``$\supseteq$'' in (\ref{eq:AG}) follows.

\smallskip

Let $T'$ be the smallest connected subgraph of $T$ containing all $t$ such that $B_t\subseteq V(G)$.
First observe that for every $i \in [m]$,
$W_i$ induces a clique in $G$ and so there exists $t \in V(T)$ such that $W_i \subseteq B_t$. Since $B_t$ and $W_i$ have both size $k+1$, we have $B_t=W_i \subseteq V(G)$.
This proves that $V(G) = \bigcup_{i \in [m]} W_i \subseteq \bigcup_{t \in V(T')} B_{t}$.

\smallskip

We now show that $T'$ is the path between $t_{G,1}$ and $t_{G,2}$ in $T$.
To do so, consider a leaf $s$ of $T'$, and let $s'$ be the neighbour of $s$ in $T'$. 
By minimality of $T'$, $B_s \subseteq V(G)$.
Now, by assumption on $(T,(B_t)_{t \in V(T)})$, there is a vertex $w$ in $B_s \setminus B_{s'}$.
Then, $N_G(w) \subseteq B_s$, and so $d_G(w) \leq k$.
This implies that $w \in \{u,v\}$.
As a consequence, $N_G(w) = B_s \in \{B_{t_{1}}, B_{t_{\ell}}\}$.
Since no bag appears twice in $(T,(B_t)_{t \in V(T)})$,
we deduce that $s \in \{t_{1},t_{\ell}\}$.
This proves that $T'$ has exactly two leaves which are $t_{G,1}$ and $t_{G,2}$,
and so $T'$ is the path between $t_{1}$ and $t_{\ell}$ in $T$.
In particular, $V(G) \subseteq B_{t_1} \cup \dots \cup B_{t_\ell}$.
Moreover, for $i\in \{2,\dots,m-1\}$, any vertex $w$ in $B_{t_i}\setminus (B_{t_{i-1}}\cup B_{t_{i+1}})$ is such that $d_G(w) \leq k$. 
Hence $w \in \{u,v\}$, and so $w \in B_{t_{1}} \cup B_{t_{\ell}}$.
It follows that $V(G) \subseteq B_{t_1} \cup B_{t_\ell} \cup (\bigcup_{i=1}^{\ell-1} B_{t_i} \cap B_{t_{i+1}})$,
which proves \eqref{eq:AG}.

The exact same arguments apply to $G'$, so $V(G)=V(G')$.
Moreover, $G$ and $G'$ are maximal graphs of pathwidth at most $k$, 
so $G = U[V(G)] = U[V(G')] = G'$. 
This shows that $G$ and $G'$ are isomorphic, as desired.
\end{proof}
In particular, for $k\geq 2$, if a graph $U$ of treewidth $k$ contains all simple $k$-paths on $n$ vertices, then by the two lemmas above
\[
|V(U)|\geq \sqrt{|S_{n,k}|}\geq k^{\frac12(n-2k-2)}.
\]

Recall that the graphs of pathwidth $1$ are exactly the graphs obtained by adding pendant vertices to unions of paths. Thus the graph obtained by adding $n-1$ leaves to each vertex of a path on $2n-1$ vertices contains all $n$-vertex graphs of pathwidth $1$ as induced subgraphs. 
Since the treewidth of a graph is at most the pathwidth of a graph, we obtain the following results for pathwidth and treewidth at least $2$.
\begin{corollary}\label{cor:lmtw}
Let $k\geq 2$ and $n\geq k+4$ be integers.
If a graph $U$ of treewidth at most $k$ contains every $n$-vertex graph of pathwidth at most $k$ as subgraph, then $U$ has at least $k^{\frac12(n-2k-2)}$ vertices. In particular, 
faithful subgraph-universal graphs for the classes $\mathcal{G}_{\mathrm{pw}\le k}$ and $\mathcal{G}_{\mathrm{tw}\le k}$ both have order at least  $n\mapsto k^{\frac12(n-2k-2)}$. 
\end{corollary}
We remark that the order of growth $2^{\Theta(n\log k)}$ is optimal for faithful subgraph-universal graphs for the class of graphs of pathwidth $k$: there is a constant $c$ such that for all $k,n$, the number $n$-vertex edge-maximal pathwidth $k$ graphs is at most $k^{cn}$ (and the disjoint union of the edge-maximal graphs in a class forms a faithful subgraph-universal graph for that class). More generally, for every proper minor-closed class $\mathcal{G}$, there is a constant $c$ such that $\mathcal{G}$ contains at most $2^{cn}$ (unlabelled) graphs on $n$ vertices \cite{BNMST24}, so the disjoint union of all such graphs (which forms an induced-universal graph for the class) contains at most $n2^{cn}=2^{O(n)}$ vertices.

\medskip

The lower bounds of Corollary \ref{cor:lmtw} also extend even if the graph is allowed to be $K_t$-minor-free, although we obtain a worse dependence on $t$ here.
To prove this, we need the following notation and lemma.
For a graph $G$ and an integer $t$, we let $G^{+t}$ denote the graph obtained from $G$ by adding $t$ universal vertices (that is, vertices inducing a clique and adjacent to all vertices of $G$). Note that for fixed $t$, the map $G\mapsto G^{+t}$ is a bijection (it is important to emphasize here that we consider unlabelled graphs throughout, that is graphs up to isomorphisms).
For a graph class $\mathcal{G}$, we denote by $\mathcal{G}^{+t}$ the class $\{G^{+t}:G\in \mathcal{G}\}$. Note that by the previous remark, $|\mathcal{G}|=|\mathcal{G}^{+t}|$ for any integer $t$. 
\begin{lemma}[Universal vertex removal lemma]
\label{lem:universal_vx_removal_lemma}
    Let $\mathcal{G}$ be a class of graphs, let $t$ be an integer,  and let $U$ contain all graphs in $\mathcal{G}^{+t}$ as subgraphs. Then $U$ contains a subgraph $U'$ which contains at least $|\mathcal{G}|/{\binom{|V(U)|}{t}}$ graphs of $\mathcal{G}$ as subgraphs, as well as a clique on $t$ vertices in $U-V(U')$ which is adjacent to all vertices of $U'$. 
\end{lemma}
\begin{proof}
For each graph $G\in \mathcal{G}$, consider a copy $H$ of $G^{+t}$ in $U$ and more particularly let
$S_G\subseteq V(U)$ be a set of $t$ universal vertices of $H$. By the pigeonhole principle, there exists a $t$-element vertex subset $S^*\subseteq V(U)$ such that $S^*=S_G$ for at least $|\mathcal{G}|/{\binom{|V(U)|}{t}}$ graphs $G\in \mathcal{G}$.

Let $U'$ denote the subgraph of $U$ induced by the vertices that are adjacent to all vertices of $S^*$ (in particular, $V(U')$ is disjoint from $S^*$). For each graph 
$G\in \mathcal{G}$ such that $S^*=S_G$, $S^*$ is adjacent to all the vertices of a copy of $G$ in $U$, and thus $U'$ contains $G$
as a subgraph, as desired.
\end{proof}

We now prove a version of \Cref{cor:lmtw} which holds even if the subgraph-universal graphs are only required to exclude $K_{t+2}$ as a minor.

\begin{theorem}\label{thm:pw_exact_lower_bound}
Let $n>  t\geq 2$ be integers.
If a $K_{t+2}$-minor-free graph $U$ contains every $n$-vertex graph of pathwidth at most $t$ as subgraph, then $U$ has at least $2^{(n-t-4)/t}$ vertices.
\end{theorem}
\begin{proof}
For $2\leq t< n\leq 5$, we have $2^{\frac1{t}(n-t)}\leq 2^{\frac12(5-2)}\leq 3\leq n$, and since $U$ needs to contain $n$ vertices to contain an $n$-vertex graph, the lower bound certainly holds. So we may henceforth assume $n\geq 6$.

Suppose that the graph $U$ is $K_{t+2}$-minor-free and contains every $n$-vertex graph of pathwidth at most $t$ as subgraph. Let $\mathcal{G}=\mathcal{S}_{n-t+2,2}$. That is, $\mathcal{G}$ contains all simple 2-paths on $n-t+2$ vertices.
For each graph from $\mathcal{G}$, adding $t-2$ universal vertices gives an $n$-vertex graph of pathwidth at most $t$, which is a subgraph of $U$. So by \cref{lem:universal_vx_removal_lemma} (applied with $t-2, ~U$ and $\mathcal{G}$), $U$ has a subgraph $U'$ which contains all graphs from a subclass $\mathcal{G}'\subseteq\mathcal{G}$ as subgraph, where 
\[
|\mathcal{G}'|\geq \frac{|\mathcal{G}|}{\binom{|V(U)|}{t-2}}\ge  \frac{|\mathcal{G}|}{|V(U)|^{t-2}},
\]
such that $U$ also has $t-2$ vertices complete to $U'$. In particular, since $U$ has no $K_{t+2}$-minor, $U'$ has no $K_4$-minor and so has treewidth at most $2$. 
We apply  \cref{lem:pw2_lb}, the displayed equation above and  \cref{lem:num_simple_pwk_graphs} to see that 
\[
|V(U)|^2\geq |V(U')|^2\geq |\mathcal{G}'|\geq |\mathcal{G}|/|V(U)|^{t-2} \geq 2^{(n-t+2-6)}/|V(U)|^{t-2}.
\]
By rearranging this inequality, we get 
$|V(U)|^t\geq 2^{n-t-4}$ and thus $|V(U)|\geq 2^{(n-t-4)/t}$.
\end{proof}
Since every graph of pathwidth at most $t$ is $K_{t+2}$-minor-free, the following result, which was suggested in \cite{BILOSW24}, is immediate from the theorem.
\begin{corollary}
\label{lem:cor_kt_lowerbound}
    For any integer $t\geq 4$,
faithful subgraph-universal graphs for the class of $K_t$-minor-free graphs have order  at least $n\mapsto 2^{(n-t-2)/(t-2)}$.
\end{corollary}

Since simple 3-paths are planar, the proof of \Cref{thm:pw_exact_lower_bound} also immediately implies the following slightly stronger version of \Cref{lem:cor_kt_lowerbound}, when $t=5$.

\begin{corollary}
\label{lem:cor_planar_k5_lowerbound}
    Any $K_5$-minor-free graph containing all $n$-vertex planar graphs as subgraph has  order  at least $n\mapsto 2^{(n-7)/3}$.
\end{corollary}

\section{Exponential lower bound for near-faithful minor-free universal graphs}\label{sec:lbgrid}

The main results in this section is \Cref{thm:Kt free main}, which gives a lower bound of $2^{\Omega_t(n)}$ on the order of a $K_t$-minor-free graph containing all planar $n$-vertex graphs as subgraph. In this result, the implicit multiplicative constant in the exponent is a polynomial in $t$.

\subsection{\texorpdfstring{$K_t$}{Kt}-minors in grid with jumps}

This subsection aims to prove  \cref{Kt double grid one interior bis} stated below. Informally, this lemma allows us to construct clique minors in a planar triangulation with a few extra edges.

Let us first introduce some notation. 
For every positive integers $a,b$, the \emph{$a \times b$ grid} is the graph with vertex set $[a] \times [b]$ and edge set $\{(x,y)(x,y+1) \mid x \in [a], y \in [b-1]\} \cup \{(x,y)(x+1,y) \mid x \in [a-1], y \in [b]\}$.
For every $i \in [b]$, the $i^{\text{th}}$ row of this grid is the set of vertices $\{(x,i) \mid x \in [a]\}$,
and for every $j \in [a]$, the $j^{\text{th}}$ column of this grid is the set $\{(j,y) \mid y \in [b]\}$.\footnote{The coordinates are chosen so that rows and columns can be imagined in an $xy-$coordinate system.}

\smallskip

A \emph{triangulated $a \times b$ grid} is a spanning supergraph $G$ of the $a \times b$ grid 
such that for every $x \in [a-1], y \in [b-1]$, 
exactly one of the pairs $(x,y)(x+1,y+1)$ and $(x+1,y)(x,y+1)$ is an edge of $G$. 

\smallskip

If $G$ is a triangulated $a\times b$ grid, we write $R_i(G)=\{(x,i):x\in [a]\}$ for the vertices in the $i$th row and $C_j(G)=\{(j,y):y\in [b]\}$ for vertices in the $j$th column. 
When $G$ is clear from the context, we will simply write $R_i$ for $R_i(G)$ and $C_j$ for $C_j(G)$.

\medskip

A \emph{jump}\footnote{Note that our definition of a jump is slightly different from the definition of a jump in \cite{huynh2021universality}, but the spirit is the same.} in a graph $G$ is a non-edge in $G$ (i.e., a pair of non-adjacent vertices). 
Our main goal is to show that adding the edges corresponding to sufficiently many jumps to a triangulated grid will create a large complete minor. 
Since additional jumps near the boundary of the grid are less helpful, we will only be interested in jumps between vertices that are not too close to the boundary. 

For integers $i,t,\ell,\ell'$, we call row $R_i$ of a triangulated $\ell\times \ell'$ grid \textit{$t$-internal} if $i\in \{t+1,t+2,\dots,\ell'-t\}$. Similarly, we call column $C_j$ \textit{$t$-internal} if $j\in \{t+1,t+2,\dots,\ell-t\}$. We call a vertex $v=(x,y)$ \textit{$t$-internal} if its column and row are $t$-internal.

 \begin{figure}[htb]
    \centering
    \includegraphics[scale=1]{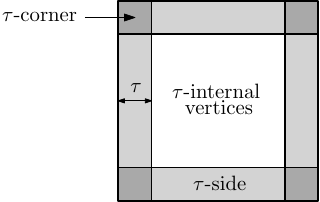}
    \caption{Dividing a grid into four corners, four sides, and internal vertices.}
    \label{fig:regions}
\end{figure}

\medskip

For an integer $\tau < \ell/2$, it will be convenient to divide each $\ell\times \ell$ grid $G$ into 9 regions (see Figure \ref{fig:regions} for an illustration): 
\begin{itemize}
\item the $\tau$-internal vertices,
    \item the four $\tau$-corners, consisting of the vertices at distance at most $\tau$ from a horizontal edge of the boundary and from a vertical edge of the boundary, and
    \item the four $\tau$-sides, consisting of the vertices at distance at most $\tau$ from the boundary, but which are not located in a $\tau$-corner.
    \end{itemize}

Note that we consider that these definitions also apply to triangulated grids (but the distance function is that of the grid, not the triangulated grid, so that the regions are the same as that described above regardless of the additional edges).

\medskip

Given a triangulated grid $G$ and a set of jumps $M$, we denote by $G\cup M$ the graph with vertex set $V(G)$ and edge set $E(G)\cup M$. 
We will need the following folklore lemma, whose proof is given in \Cref{sec:app} for  completeness.

\begin{restatable}{lemma}{ktgridoneinteriorbis}
    \label{Kt grid one interior bis}
    There is a polynomial function $f_{\ref{Kt grid one interior bis}}\colon \mathbb{N}_{>0} \to \mathbb{N}_{>0}$ such that the following holds.
    Let $t$ be a positive integer,
    let $\ell,\ell'$ be integers with $\ell,\ell'\geq 2f_{\ref{Kt grid one interior bis}}(t)$. Let $G$ be a triangulated $\ell \times \ell'$ grid.
    For every set $M$ of pairwise disjoint jumps of $G$,
    if
    \begin{enumerate}
        \item for every $uv \in M$, $u$ or $v$ is $f_{\ref{Kt grid one interior bis}}(t)$-internal; and
        \item $|M| \geq f_{\ref{Kt grid one interior bis}}(t)$;
    \end{enumerate}
    then
    $K_t$ is a minor of $G \cup M$.
\end{restatable}

Consider two disjoint triangulated $\ell\times \ell$ grids $G_1$ and $G_2$, and let $H$ be a planar triangulation obtained by 
\begin{enumerate}
    \item adding an edge between the vertex $(i,j)$ of $G_1$ and the vertex $(i,j)$ of $G_2$, for any pair $i,j$ such that $(i,j)$ lies on the outerface of $G_1$ (and equivalently $G_2$).
    \item adding an edge in each quadrangular face of the resulting plane graph (all such faces include boundary vertices of $G_1$ and $G_2$).
\end{enumerate} The graph $H$, which is a planar triangulation on $2\ell^2$ vertices, is called a \emph{triangulated $\ell\times \ell$ double grid}. See Figure \ref{fig:doublegrid} for an illustration with $\ell=3$.

\begin{figure}[htb]
    \centering
    \includegraphics[scale=1.8]{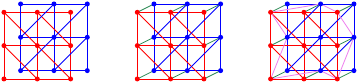}
    \caption{A triangulated $3\times 3$ double grid (right) obtained from two disjoint triangulated $3\times 3$ grids (left) by first adding a perfect matching between their boundaries (center) and then triangulating the 8 quadrangular faces created in the process. We view the resulting planar triangulation as being embedded on the sphere (or rather a cube).}
    \label{fig:doublegrid}
\end{figure}

We now use \Cref{Kt grid one interior bis} to prove the following variant for triangulated double grids, where the endpoint of the jumps are not required to be internal.

\begin{lemma}
    \label{Kt double grid one interior bis}
    There is a polynomial function $f_{\ref{Kt double grid one interior bis}}\colon \mathbb{N}_{>0} \to \mathbb{N}_{>0}$ such that the following holds.
    Let $t$ be a positive integer and
    let $\ell$ be an integer with $\ell\geq 2f_{\ref{Kt double grid one interior bis}}(t)$. Let $H$ be a triangulated $\ell\times \ell$ double grid.
    For every set $M$ of pairwise disjoint jumps of $H$,
    if $|M| \geq f_{\ref{Kt double grid one interior bis}}(t)$,
    then
    $K_t$ is a minor of $H\cup M$.
\end{lemma}

\begin{proof}
Let $\tau=f_{\ref{Kt grid one interior bis}}(t)$ and set $f_{\ref{Kt double grid one interior bis}}(t)=8\tau^2+4\tau$. 

Let $G_1$ and $G_2$ be the two triangulated $\ell\times \ell$ grids forming the triangulated $\ell\times \ell$ double grid~$H$. Note that in the unique embedding of $H$ in the sphere (up to reflection),  there is a natural cyclic ordering $e_1,\ldots,e_{8\ell-7}=e_1$
on the $8\ell-8$ edges connecting $G_1$ to $G_2$ in $H$. We let $e_1$ be the edge connecting the two vertices of coordinates $(1,1)$ in $G_1$ and $G_2$. We now define four spanning subgraphs $H_1,\ldots,H_4$ of $H$. For each $1\le i\le 4$, we let $H_i$ be the subgraph obtained from $H$ be deleting all edges $e_j$, with $j\in [2(i-1)(\ell-1)+1,2(i+2)(\ell-1)]$, where integers are considered modulo $8\ell-8$. See Figure \ref{fig:Hi} for an illustration.

\begin{figure}[htb]
    \centering
    \includegraphics[scale=1]{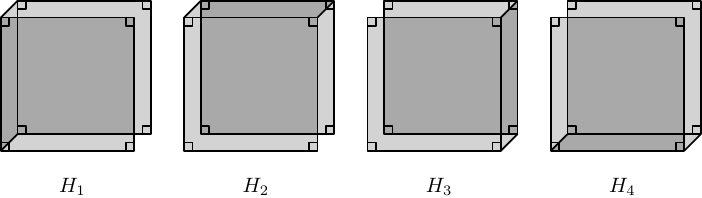}
    \caption{The support of the four graph $H_1,\ldots,H_4$. Every vertex which is not in a corner is far from the boundary in at least one of these four graphs.}
    \label{fig:Hi}
\end{figure}

Note that each $H_i$ is a triangulated $\ell \times 2\ell$ grid. Moreover,  each vertex $v$ of $H$ is either in a $\tau$-corner of $G_1$ or $G_2$, or $v$ is $\tau$-internal in some $H_i$ for $1\le i\le 4$. Observe that each $\tau$-corner of $G_1$ or $G_2$ contains at most $\tau^2$ vertices, and since all jumps from $M$ are disjoint, at most $8\tau^2$ jumps from $M$ have an endpoint in a $\tau$-corner of $G_1$ or $G_2$. It follows that at least $f_{\ref{Kt double grid one interior bis}}(t)-8\tau^2=4\tau$ jumps of $M$ do not have any endpoint in a $\tau$-corner, and are thus $\tau$-internal in some $H_i$. By the pigeonhole principle, there is $1\le i\le 4$ and a subset $M'\subseteq M$ of size at least $\tau$ such that each jump of $M'$ has an endpoint which is $\tau$-internal in $H_i$. By  \Cref{Kt grid one interior bis}, $H_i\cup M'$ (and thus $H\cup M$) contains $K_t$ as a minor.
\end{proof}

\subsection{Creating the jumps}
Together with \Cref{Kt double grid one interior bis} above, we will need the following two classical results.

\begin{lemma}[\cite{Kos84,Tho84}]\label{lemma:density Kt minor-free}
    There is a polynomial function $f_{\ref{lemma:density Kt minor-free}} \colon \mathbb{N}_{>0} \to \mathbb{N}_{>0}$ such that the following holds.
    Let $t$ be a positive integer.
    For every graph $G$,
    if $K_t$ is not a minor of $G$, then
    \[
    |E(G)| \leq f_{\ref{lemma:density Kt minor-free}}(t) |V(G)|.
    \]
\end{lemma}
Given a set $X$, we write $2^X$ for the power set of $X$.
Let $\mathcal{F}\subseteq 2^{X}$ be a given set system on the ground set $X$. The \textit{VC-dimension} of $\mathcal{F}$ is the supremum taken over the integers $d\geq 0$ for which there is a set $S\subseteq X$ of size $d$ which is \emph{shattered}, that is such that $\{Y\cap S:Y\in \mathcal{F}\}=2^S$.
\begin{lemma}[Sauer-Shelah \cite{Sau72,She72}]
\label{lem:SauerShelah}
If the VC-dimension of $\mathcal{F}\subseteq 2^{[n]}$ is at most $d$, then $|\mathcal{F}|\leq \sum_{i=0}^d \binom{n}i
\leq  2n^d$. 
\end{lemma}
We will combine \Cref{Kt double grid one interior bis} with  \Cref{lemma:density Kt minor-free,lem:SauerShelah} above to prove the following result,
which informally says that a $K_t$-minor-free graph on $2\ell^2$ vertices 
cannot contain too many triangulated $\ell\times \ell$ double grids as spanning subgraph. Recall that given a graph $U$ and a subset $S$ of edges of $U$, $U[S]$ denotes the subgraph of $U$ whose vertex set is the set of endpoints of the edges in $S$, and whose edge set is $S$.

\begin{lemma}\label{lemma:only_a_few_grids_on_the_same_vertex_set}
    There is a polynomial function $f_{\ref{lemma:only_a_few_grids_on_the_same_vertex_set}} \colon \mathbb{N}_{>0} \to \mathbb{N}_{>0}$ such that the following holds.
    Let $t, \ell\ge 2$ be integers.
    Let $U$ be a $K_t$-minor-free graph on $2\ell^2$ vertices.
    Then there are at most     $\ell^{f_{\ref{lemma:only_a_few_grids_on_the_same_vertex_set}}(t)}$ edge sets $S\subseteq E(U)$ such that $U[S]$ is a spanning subgraph of $U$ isomorphic to a triangulated $\ell\times \ell$ double grid.
\end{lemma}

\begin{proof}
Set $d=15\cdot f_{\ref{Kt double grid one interior bis}}(t)$, and let $\mathcal{G}\subseteq 2^{E(U)}$ be the family of subsets of edges of $U$ such that $U[S]$ is a spanning subgraph of $U$ isomorphic to a triangulated $\ell\times \ell$ double grid. 
    We define $f_{\ref{lemma:only_a_few_grids_on_the_same_vertex_set}}(t)$ such that
    \[
         \ell^{f_{\ref{lemma:only_a_few_grids_on_the_same_vertex_set}}(t)}\geq 2(f_{\ref{lemma:density Kt minor-free}}(t)\cdot 2 \ell^2)^{d}.
    \]
Assume for the sake of contradiction that $|\mathcal{G}| > \ell^{f_{\ref{lemma:only_a_few_grids_on_the_same_vertex_set}}(t)}$. By \Cref{lemma:density Kt minor-free}, $|E(U)|\le f_{\ref{lemma:density Kt minor-free}}(t)\cdot 2\ell^2$ and thus $|\mathcal{G}|  > 2|E(U)|^{d}$.    
    By \cref{lem:SauerShelah},
    there exists $A \subseteq E(U)$ with $|A|\geq d+1$ such that for every $B \subseteq A$, there exists $E_B \in \mathcal{G}$ such that 
    $E_B \cap A = B$. 
    Let $G_1 = U[E_A]$ and $G_2 = U[E_\emptyset]$.
    Since $G_1$ has maximum degree at most $8$,
    there exists a matching $M$ included in $A$ of size at least
    $\frac{|A|}{15} \geq f_{\ref{Kt double grid one interior bis}}(t)$. The set $M$ corresponds to a set of pairwise disjoint jumps in $G_2$, and it thus follows from \Cref{Kt double grid one interior bis} that $G_2\cup M$ contains a $K_t$-minor, and thus $U$ also contains a $K_t$-minor, which is a contradiction.
 \end{proof}

We are now ready to prove the main result of this section.

\begin{theorem}\label{thm:Kt free main}
    There is a polynomial function $f_{\ref{thm:Kt free main}} \colon \mathbb{N}_{>0} \to \mathbb{N}_{>0}$ such that the following holds.
    Let $t,\ell\ge 2$ be  integers.
    If $U$ is a $K_t$-minor-free graph containing every triangulated $\ell\times \ell$ double grid as subgraph, then
    \[
        |V(U)| \geq 2^{\ell^2/f_{\ref{thm:Kt free main}}(t)}.
    \]
     In particular, for every integer $t\geq 5$ there exists a constant $C_t > 0$ such that for every integer $n\ge 2$, every $K_{t}$-minor-free graph containing every $n$-vertex planar graph as subgraph has at least $2^{C_t n}$ vertices.
\end{theorem}
\begin{proof}
We set $d=48 f_{\ref{Kt double grid one interior bis}}(t)+48$, and $f_{\ref{thm:Kt free main}}(t)=16 \cdot \max(d,f_{\ref{lemma:only_a_few_grids_on_the_same_vertex_set}}(t)^2)$. Let $U$ be a $K_t$-minor-free graph containing all triangulated  $\ell\times \ell$ double grids, and assume for the sake of contradiction that  $N=|V(U)|<2^{\ell^2/f_{\ref{thm:Kt free main}}(t)}$. Since $N\ge 2$, we can assume that $\ell \geq \sqrt{f_{\ref{thm:Kt free main}}(t)}\ge 3$.

Observe that there are at least $2^{2\ell^2-2}/24\ell^2$ non-isomorphic triangulated  $\ell\times \ell$ double grids: starting from the planar quadrangulation obtained from two disjoint copies of the $\ell \times \ell$ grid by adding a perfect matching between their boundaries, we obtain $2^{2\ell^2-2}$ distinct ways to complete this quadrangulation into a triangulated $\ell \times \ell$ double grid (as in each of the $2\ell^2-2$ quadrangular faces, we can add either of the two diagonals). It was proved by Weinberg \cite{Wein66} that any 3-connected planar graph on $m$ edges has at most $4m$ automorphisms (see also \cite{HT66} for a short proof). It thus follows that there are at least $2^{2\ell^2-2}/24\ell^2$ non-isomorphic triangulated  $\ell\times \ell$ double grids, as desired. It can be checked  that as $\ell\ge 3$, we have $2^{2\ell^2-2}/24\ell^2\ge 2^{\ell^2/2}$.

\medskip

By \Cref{lemma:only_a_few_grids_on_the_same_vertex_set}, for each set $A\subseteq V(U)$ of size $2\ell^2$, at most $2^{f_{\ref{lemma:only_a_few_grids_on_the_same_vertex_set}}(t)\log \ell}$ triangulated  $\ell\times \ell$ double grids can be accounted for by $U[A]$. Hence, there is a family $\mathcal{G}\subseteq 2^{V(U)}$ of  $2^{\ell^2/2-f_{\ref{lemma:only_a_few_grids_on_the_same_vertex_set}}(t)\log \ell}$ vertex subsets of size $2\ell^2$ of $U$ that each span a triangulated  $\ell\times \ell$ double grid. As $\ell\ge 4 f_{\ref{lemma:only_a_few_grids_on_the_same_vertex_set}}(t)$, we have $\ell^2/4 \ge f_{\ref{lemma:only_a_few_grids_on_the_same_vertex_set}}(t)\ell\ge f_{\ref{lemma:only_a_few_grids_on_the_same_vertex_set}}(t) \log \ell$ and thus $\ell^2/2-f_{\ref{lemma:only_a_few_grids_on_the_same_vertex_set}}(t)\log \ell\ge  \ell^2/4$.

    We obtain
    \[
    |\mathcal{G}| =2^{\ell^2/2 - f_{\ref{lemma:only_a_few_grids_on_the_same_vertex_set}}(t) \log \ell}\ge 2^{\ell^2/4}\ge 2^{2\ell^2d/f_{\ref{thm:Kt free main}}(t)}\ge 2\big(2^{\ell^2/f_{\ref{thm:Kt free main}}(t)}\big)^{d}>2N^{d},
    \]
where the second inequality follows from the fact that  $f_{\ref{thm:Kt free main}}(t)\ge 8d$ and the third follows from the fact that $\ell \geq \sqrt{f_{\ref{thm:Kt free main}}(t)}$ and thus $\ell^2 d/f_{\ref{thm:Kt free main}}(t)\ge d \ge 1$. The final inequality follows from our initial assumption that $N=|V(U)|<2^{\ell^2/f_{\ref{thm:Kt free main}}(t)}$.

    It then follows from \Cref{lem:SauerShelah} that
    the VC-dimension of $\mathcal{G}$ is at least $d+1$. 
    Hence there exists $A \subseteq V(U)$ of size at least $d+1$
    such that for every $B \subseteq A$, there exists a  $V_B\in \mathcal{G}$ such that
    \[
        V_B \cap A = B.
    \]
    In particular, $A\subseteq V_A\in \mathcal{G}$.    
    Let $G_1$ be a triangulated $\ell \times \ell$ double grid spanning $U[V_A]$.

\medskip
Note that the vertices of every triangulated grid can be covered by  four (not necessarily disjoint) induced paths: two horizontal paths covering the top and bottom rows of the grid, one path containing all vertices in the odd numbered columns (except for possibly those on the top and bottom rows, see \Cref{fig:building_B_and_Q v2} for an illustration) and the fourth path for the vertices in the even numbered columns.

This implies that every triangulated \emph{double} grid can be covered by eight induced paths. By the pigeonhole principle, this means that $G_1$ contains an induced path $P$ containing at least $|A|/8$ vertices from $A$.

We select $k= \lfloor|V(P)|/3\rfloor\ge |A|/24-1$ vertex-disjoint subpaths $P_1,\ldots,P_k$ of $P$ such that for each $1\le i \le k$, the two endpoints of $P_i$ are in $A$ and exactly one internal vertex of $P_i$ lies in $A$ (so each $P_i$ contains precisely 3 vertices of $A$). Let $B\subseteq A$ be the set of endpoints of the paths $P_i$, $1\le i \le k$. Note that every path $P_i$ contains a vertex of $A\setminus B$.

We then consider a set $V_B\in \mathcal{G}$ such that $V_B\cap A=B$ and let $G_2$ be a triangulated $\ell \times \ell$ grid spanning $U[V_B]$.
We will show that $G_1 \cup G_2$ contains a $K_t$-minor, hence contradicting the fact that $U$ is $K_t$-minor-free. 

\medskip

Let us construct two families of paths $\mathcal{Q}_1$ and $\mathcal{Q}_2$
    satisfying the following properties:
    \begin{enumerate}
        \item  the members of $\mathcal{Q}_1$ are pairwise disjoint paths in $G_{2}$ internally disjoint from $V(G_1)$, and for any path $Q\in \mathcal{Q}_1$, the endpoints of $Q$ are non-adjacent in $G_1$ (so the pairs of endpoints of the paths from $\mathcal{Q}_1$ form pairwise disjoint jumps in $G_1$);
        \item  the members of $\mathcal{Q}_2$ are pairwise disjoint paths in $G_{1}$ internally disjoint from $V(G_2)$, , and for any path $Q\in \mathcal{Q}_2$, the endpoints of $Q$ are non-adjacent in $G_2$ (so the pairs of endpoints of the paths from $\mathcal{Q}_2$ form pairwise disjoint jumps in $G_2$).
    \end{enumerate}    

    \begin{figure}
    \centering
    \includegraphics[scale=1]{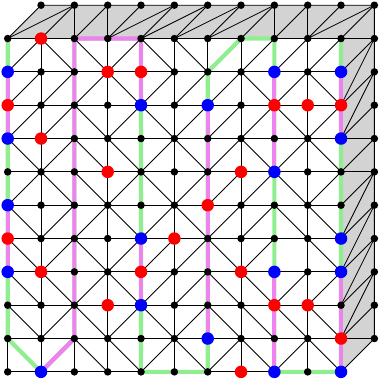}
    \caption{Illustration for the proof of \Cref{thm:Kt free main}.
        For the sake of clarity we only depict one of the two triangulated grids forming $G_1$ and part of the boundary of the second grid.
        The red and blue vertices represent the shattered set $A$,
        and the blue vertices are the members of the subset $B$.
        Along a path $P$ in the grid (in bold green and violet) we build disjoint subpaths of $P$ between vertices in $B$ via a vertex in $A\setminus B$.
        These subpaths are depicted in violet.
        This ensures that every such path contains a path in $V(G_1)$ between two vertices $u,v\in V(G_2)$
        of length at least $2$ internally disjoint from $V(G_2)$.
        If $u,v$ are non-adjacent in $G_2$, then this creates a jump in $G_2$,
        and we put it in $\mathcal{Q}_2$.
        Otherwise, the edge $uv$ creates a jump in $G_1$, that we put in $\mathcal{Q}_1$.
        At the end of this process, either $\mathcal{Q}_1$ or $\mathcal{Q}_2$ is large, which will imply the existence of a $K_t$-minor.}
    \label{fig:building_B_and_Q v2}
\end{figure}

    For any $1\le i \le k$, we do the following. By definition, $P_i$ intersects $A \setminus B$,
    and thus $P_i$ contains a subpath $Q_i$ between two vertices  $x,y \in V(G_2)$ such that $Q_i$ contains at least one internal vertex, and such that all internal vertices of $Q_i$ lie in  $V(G_1)\setminus V(G_2)$.
    If the endpoints $x,y$ of $Q_i$ are non-adjacent in $G_2$, then add $Q_i$ to $\mathcal{Q}_2$.
    Otherwise add the path consisting of the single edge $xy$ to $\mathcal{Q}_1$. See Figure \ref{fig:building_B_and_Q v2} for an illustration.
    Observe that the two properties above are maintained in both cases. 

    \medskip

    Since $|\mathcal{Q}_1 \cup \mathcal{Q}_2|=k\ge |A|/24-1$, there exists $a \in \{1,2\}$ such that $|\mathcal{Q}_a| \geq |A|/48-1\ge d/48-1=f_{\ref{Kt double grid one interior bis}}(t)$. 
    Fix such $a\in \{1,2\}$, and let $M$ be the family of all the pairs of endpoints of paths in $\mathcal{Q}_a$.
    Since $\mathcal{Q}_a$ is a family of pairwise disjoint paths internally disjoint from $V(G_a)$,
    $G_a \cup M$ is a minor of $U$.
    Moreover, $M$ is disjoint from $E(G_a)$ and since $|M|\ge f_{\ref{Kt double grid one interior bis}}(t)$, we conclude by \Cref{Kt double grid one interior bis} that $K_t$ is a minor of $G_a \cup M$, and so of $U$.
\end{proof}

\section{Conclusion}\label{sec:ccl}

\subsection{Finite versus infinite}

One of our main motivations was to obtain a finite version of the result of Huynh, Mohar, Šámal, Thomassen and Wood \cite{huynh2021universality} (\cref{thm:h21} in the introduction), stating that a countable graph containing all  countable planar graphs as subgraph has an infinite clique minor. A natural question is whether there is a direct connection between the infinite and finite versions of the problem. While it does not seem to us that one result can be quickly deduced from the other, we should note that our approach for producing $K_t$-minors out of short jumps in \cref{lem:slocal} can be used greedily in the infinite case to produce arbitrarily large clique minors, because in the infinite case all jumps can be considered as short. This alternative approach does not directly produce infinite clique minors as in \cite{huynh2021universality} (additional compactness arguments are needed), but this highlights the fact that the finite version of the problem contains a number of challenges that do not appear in the infinite version (such as boundary effects and the existence of long jumps), regardless of any quantitative aspects.

An anonymous reviewer asked about a possible meta-conjecture relating the existence of countable faithful universal graphs to that of polynomial size faithful universal graphs for general graph classes. We now show that no such relation exists in general, even for classes that are monotone and closed under disjoint union (these restrictions are natural and they imply the existence of faithful universal graphs, obtained by taking the disjoint union of all graphs in the class). 

In one direction, consider triangle-free graphs. The Henson graph \cite{henson1971family} is a countable triangle-free graph containing all countable triangle-free graphs as induced subgraphs (and thus as subgraphs). Note that any triangle-free graph $U_n$ containing all $n$-vertex triangle-free graphs as subgraphs must contain all $n$-vertex maximal triangle-free graphs as induced subgraphs. There are at least $2^{n^2/8}$ distinct maximal triangle-free graphs on $n$ vertices \cite{BP11}, and thus at least $2^{(1/8-o(1))n^2}$ non-isomorphic ones. Therefore, we have \[2^{(1/8-o(1))n^2}\le \binom{|V(U_n)|}{n}\le |V(U_n)|^n,\] and thus  $U_n$ must have size at least $2^{n/8-o(n)}$.

In the other direction, for every finite binary sequence $b=(b_1,\ldots,b_k)$, where $b_i \in \{0,1\}$, define $G_b$ to be the graph obtained from the path $v_1,\ldots,v_{2^k}$ by attaching, at each vertex $v_{2^{i}}$, a triangle if $b_i=0$, and a copy of $P_2$ otherwise (see Figure \ref{fig:counterexample_meta} for an example). Now let $\mathcal{C}$ be the closure of the family $\{G_b \mid k \in \mathbb{N},\ b \in \{0,1\}^k\}$ under subgraph and disjoint union. It is easy to see that $\mathcal{C}$ admits faithful universal graphs of size $O(n^3)$. Indeed, it suffices to take $n$ copies of the disjoint union of all graphs $G_b$ with $b \in \{0,1\}^{\lceil \log n \rceil}$.

\begin{figure}[ht]
    \centering
    \includegraphics{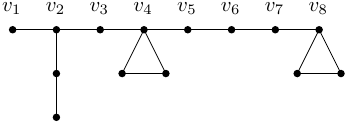} 
    \caption{Representation of $G_b$ for $b = (1,0,0).$}
    \label{fig:counterexample_meta}
\end{figure}

Let $\mathcal{D}$ be the closure of the family $\{G_b \mid  b \in \{0,1\}^\mathbb{N}\}$ under subgraph and disjoint union, where $G_b$ is defined for infinite binary sequences analogously to the finite case.
We observe that there is no countable graph $U$ in $\mathcal{D}$ which contains all countable graphs from $\mathcal{D}$ as subgraph, because $U$ can encode only countably many binary sequences, whereas the set of infinite binary sequences is uncountable.

\subsection{Open problems}

In \cref{tw upper bound}, we constructed a polynomial-size induced-universal graph \( U \) for the class of graphs with treewidth \( t \), ensuring that \( \tw(U) \le 3t+1 \). This means that the increase in treewidth from the class to the induced-universal graph is bounded by a constant factor. In contrast, for pathwidth, \cref{pw upper bound} only guarantees an upper bound where the increase in pathwidth grows quadratically. We believe this discrepancy is not merely a limitation of our technics, but may reflect a difference in behaviour between pathwidth and treewidth. This leads us to the following question:
\begin{problem}
    Is there a constant $C>0$ such that for every $k\in \mathbb{N}$, the class $\mathcal{G}_{\pw \leq k}$ admits subgraph-universal graphs of pathwidth at most $Ck$ and polynomial order $n^{O_k(1)}$?
\end{problem}
In terms of subgraph-universal $K_t$-minor-free graphs containing all $K_s$-minor-free graphs, we  settle the (approximate) order except for the case $s=4$ and $t=5,6$. 
In Corollary~\ref{cor:K5-m-free universal for K4-m-free}, we show that $K_5$-minor-free subgraph-universal graphs of quasi-polynomial order can be obtained for the class of $K_4$-minor-free graphs and in fact those universal graphs have treewidth~3. This leaves the question of whether such a graph can be obtained of polynomial size.
\begin{problem}
    Does the class of $K_4$-minor-free graphs admit a subgraph-universal graph of polynomial order which is $K_5$-minor-free? If so, can it even be chosen of treewidth~$3$? If not, what if we allow it to be $K_6$-minor-free? 
\end{problem}
We also remark that our constructions of quasi-polynomial order are only subgraph-universal. It may also be interesting to see if constructions with similar properties can be made that are induced-universal.
\medskip

Our lower bound $2^{(n-t-2)/(t-2)}$ in Corollary \ref{lem:cor_kt_lowerbound} becomes worse as $t$ increases. One would rather expect the opposite, that is, a lower bound of the form $2^{\omega_t(1)n}$.
\begin{problem}
    Is there an absolute constant $c$ such that for every integer $t$, the class of $K_t$-minor-free graphs has faithful subgraph-universal graphs of order $n\mapsto 2^{cn}$?
\end{problem}

The paper has been mostly dedicated to minor-closed classes, but interesting questions can be raised more generally for monotone or hereditary classes. The following problem was raised in \cite{BEFZ25}, in the context of local certification in distributed computing.

\begin{problem}[\cite{BEFZ25}]
    For which  graph $H$ is it the case that the class of $H$-subgraph-free graphs has faithful subgraph-universal graphs of order $n\mapsto 2^{o(n^2)}$? 
\end{problem}

Note that when $H$ is bipartite, there is a real $\varepsilon>0$ such that $H$-subgraph-free graphs on $n$ vertices have $O(n^{2-\varepsilon})$ edges and thus there are at most $2^{o(n^2)}$ $n$-vertex $H$-subgraph-free graphs. When moreover, $H$ is connected, then the disjoint union of all these graphs is also $H$-subgraph-free and in this case the class of $H$-subgraph-free graphs has faithful subgraph-universal graphs of order $n\mapsto 2^{o(n^2)}$.

\section*{Acknowledgement}

Part of this research was conducted during the ``New Perspectives in Colouring and Structure" workshop, held at the Banff International Research Station, September 29--October 4, 2024, and the Oberwolfach workshop 2502 ``Graph Theory'', January 6--10, 2025.  We are grateful to the organisers and participants for providing a stimulating research environment.  

We would also like to thank Gwena\"el Joret, whose suggestion to look at grids glued on the faces of a cube lead to a significant improvement in our results.

Finally, we thank the two reviewers for their helpful corrections and suggestions. 

\bibliographystyle{alpha}
\bibliography{biblio}

@article{GL68,
    author = {Gol'dberg, M. K. and Livshits, É. M.},
    title = {On minimal universal trees},
    journal = {Mathematical notes of the Academy of Sciences of the USSR},
    volume = {4},
    pages = {713-717},
    doi = {https://doi.org/10.1007/BF01116454},
    url = {https://doi.org/10.1007/BF01116454},
    year = {1968}
}

@INPROCEEDINGS{ElberfeldJakobyTantau10,
  author={Elberfeld, Michael and Jakoby, Andreas and Tantau, Till},
  booktitle={2010 IEEE 51st Annual Symposium on Foundations of Computer Science}, 
  title={Logspace Versions of the Theorems of {B}odlaender and {C}ourcelle}, 
  year={2010},
  volume={},
  number={},
  pages={143-152},
  doi={10.1109/FOCS.2010.21}}

@article{huynh2021universality,
  title={Universality in minor-closed graph classes},
  author={Huynh, Tony and Mohar, Bojan and {\v{S}}{\'a}mal, Robert and Thomassen, Carsten and Wood, David R.},
  journal={arXiv preprint},
  year={2021},
  note = {\href{https://arxiv.org/abs/2109.00327}{arXiv:2109.00327}}
}

@article{diestel1985some,
  title={Some remarks on universal graphs},
  author={Diestel, Reinhard and Halin, Rudolf and Vogler, Walter},
  journal={Combinatorica},
  volume={5},
  pages={283--293},
  year={1985},
  publisher={Springer}
}

@inproceedings{BILOSW24,
  author       = {Helena Bergold and
                  Vesna Irsic Chenoweth and
                  Robert Lauff and
                  Joachim Orthaber and
                  Manfred Scheucher and
                  Alexandra Wesolek},
  editor       = {Henning Fernau and
                  Philipp Kindermann},
  title        = {Subgraph-Universal Planar Graphs for Trees},
  booktitle    = {Graph-Theoretic Concepts in Computer Science - 51st International
                  Workshop, {WG} 2025, Otzenhausen, Germany, June 11-13, 2025, Revised
                  Selected Papers},
  series       = {Lecture Notes in Computer Science},
  pages        = {62--75},
  publisher    = {Springer},
  year         = {2025},
  url          = {https://doi.org/10.1007/978-3-032-11835-6\_5},
  doi          = {10.1007/978-3-032-11835-6\_5},
  bibsource    = {dblp computer science bibliography, https://dblp.org}
}

@inbook{Bodlaender1998,
  title = {Tree decompositions of small diameter},
  ISBN = {9783540685326},
  ISSN = {1611-3349},
  url = {http://dx.doi.org/10.1007/BFb0055821},
  DOI = {10.1007/bfb0055821},
  booktitle = {Mathematical Foundations of Computer Science 1998},
  publisher = {Springer Berlin Heidelberg},
  author = {Bodlaender,  Hans L. and Hagerup,  Torben},
  year = {1998},
  pages = {702–712}
}

@ARTICLE{CGC81,
     AUTHOR = {Chung, Fan R. K. and Graham, Ronald L. and Coppersmith, Don},
      TITLE = {On trees which contain all small trees},
    JOURNAL = {The Theory of Applications of Graphs},
     EDITOR = {Chartrand, Gary},
  PUBLISHER = {John Wiley {\&} Sons},
      PAGES = {265--272},
       YEAR = 1981,
}

@inproceedings{GJ22,
  author    = {Pawe{\l} Gawrychowski and
               Wojciech Janczewski},
  editor    = {Karl Bringmann and
               Timothy Chan},
  title     = {Simpler Adjacency Labeling for Planar Graphs with {B}-Trees},
  booktitle = {5th Symposium on Simplicity in Algorithms, SOSA@SODA 2022, Virtual
               Conference, January 10-11, 2022},
  pages     = {24--36},
  publisher = {{SIAM}},
  year      = {2022},
  url       = {https://doi.org/10.1137/1.9781611977066.3},
  doi       = {10.1137/1.9781611977066.3},
  bibsource = {dblp computer science bibliography, https://dblp.org}
}

@inproceedings{chung1990-separator,
 Author = {Fan R. K. Chung},
 Title = {Separator theorems and their applications},
 Pages = {17--34},
 Year = {1990},
 Language = {English},
 booktitle = {{Paths, flows, and VLSI-layout, Proc. Meet., Bonn/Ger. 1988, Algorithms Comb. 9, 17-34 (1990)}}
}

@article{chung1990-universal-induced-universal,
author = {Chung, Fan R. K.},
title = {Universal graphs and induced-universal graphs},
journal = {Journal of Graph Theory},
volume = {14},
number = {4},
pages = {443--454},
doi = {10.1002/jgt.3190140408},
year = {1990}
}

@article{chung.graham,
author = {Chung, Fan R. K. and Graham, Ronald L.},
title = {On Universal Graphs for Spanning Trees},
journal = {Journal of the London Mathematical Society},
volume = {s2-27},
number = {2},
pages = {203--211},
doi = {10.1112/jlms/s2-27.2.203},
year = {1983}
}

@article{bhatt.chung.ea,
 Author = {Sandeep N. Bhatt and Fan R. K. Chung and Frank T. Leighton and Arnold L. Rosenberg},
 Title = {Universal graphs for bounded-degree trees and planar graphs},
 FJournal = {SIAM Journal on Discrete Mathematics},
 Journal = {SIAM Journal on Discrete Mathematics},
 ISSN = {0895-4801; 1095-7146/e},
 Volume = {2},
 Number = {2},
 Pages = {145--155},
 Year = {1989},
 Publisher = {Society for Industrial and Applied Mathematics (SIAM), Philadelphia, PA}
}

@incollection{babai.chung.ea,
 Author = {Laszlo Babai and Fan R. K. Chung and Paul Erd\H{o}s and Ronald L. Graham and Joel H. Spencer},
 Title = {On graphs which contain all sparse graphs},
 BookTitle = {Theory and practice of combinatorics. A collection of articles honoring Anton Kotzig on the occasion of his sixtieth birthday},
 publisher = {Elsevier},
 Pages = {21--26},
 Year = {1982}
}

@inproceedings{BL82,
author = {Bhatt, Sandeep N. and Leiserson, Charles E.},
title = {How to assemble tree machines (Extended Abstract)},
year = {1982},
isbn = {0897910702},
publisher = {Association for Computing Machinery},
address = {New York, NY, USA},
url = {https://doi.org/10.1145/800070.802179},
doi = {10.1145/800070.802179},
booktitle = {Proceedings of the Fourteenth Annual ACM Symposium on Theory of Computing},
pages = {77–84},
numpages = {8},
location = {San Francisco, California, USA},
series = {STOC '82}
}

@article{AdjacencyLabellingPlanarJACM,
    Author =  {Vida Dujmovi\'c and Louis Esperet and Cyril Gavoille and Gwena\"el Joret and Piotr Micek and Pat Morin},
    Title = {Adjacency labelling for planar graphs (and beyond)},
    journal   = {Journal of the ACM}, 
    volume = {68},
    number = {6}, 
    Pages = {Article 42},
    Year = {2021}
}

@inproceedings{bonamy.gavoille.ea:shorter,
  author    = {Marthe Bonamy and
               Cyril Gavoille and
               Michał Pilipczuk},
  editor    = {Shuchi Chawla},
  title     = {Shorter Labeling Schemes for Planar Graphs},
  booktitle = {Proceedings of the 2020 {ACM-SIAM} Symposium on Discrete Algorithms,
               {SODA} 2020, Salt Lake City, UT, USA, January 5-8, 2020},
  pages     = {446--462},
  publisher = {{SIAM}},
  year      = {2020},
  url       = {https://doi.org/10.1137/1.9781611975994.27},
  doi       = {10.1137/1.9781611975994.27},
  bibsource = {dblp computer science bibliography, https://dblp.org},
  }

@inproceedings{gavoille.labourel:shorter,
  author    = {Cyril Gavoille and
               Arnaud Labourel},
  title     = {Shorter Implicit Representation for Planar Graphs and Bounded Treewidth
               Graphs},
  booktitle = {Algorithms -- {ESA} 2007, 15th Annual European Symposium, Eilat, Israel,
               October 8--10, 2007, Proceedings},
  pages     = {582--593},
  year      = {2007},
  url       = {https://doi.org/10.1007/978-3-540-75520-3_52},
  doi       = {10.1007/978-3-540-75520-3_52},
  bibsource = {dblp computer science bibliography, https://dblp.org}
}

@article{kannan.naor.ea:implicit,
  author    = {Sampath Kannan and
               Moni Naor and
               Steven Rudich},
  title     = {Implicit Representation of Graphs},
  journal   = {{SIAM} J. Discrete Math.},
  volume    = {5},
  number    = {4},
  pages     = {596--603},
  year      = {1992},
  url       = {https://doi.org/10.1137/0405049},
  doi       = {10.1137/0405049},
  bibsource = {dblp computer science bibliography, https://dblp.org}
}

@article{Alon17,
  author    = {Alon, Noga},
  title     = {Asymptotically optimal induced universal graphs},
  journal   = {Geometric and Functional Analysis},
  volume    = 27,
  number    = 1,
  pages     = {1-32},
  month     = feb,
  year      = 2017,
  doi       = {10.1007/s00039-017-0396-9}
}

@book{spinrad:efficient,
  author    = {Jeremy P. Spinrad},
  title     = {Efficient graph representations},
  series    = {Fields Institute monographs},
  volume    = {19},
  publisher = {American Mathematical Society},
  year      = {2003},
  url       = {http://www.ams.org/bookstore-getitem/item=fim-19},
  isbn      = {978-0-8218-2815-1},
  bibsource = {dblp computer science bibliography, https://dblp.org}
}

@phdthesis{muller:local,
  author    = {John H. Muller},
  title     = {Local Structure in Graph Classes},
  institute = {Georgia Institute of Technology},
  school    = {School of Information and Computer Science},
  month     = {March},
  year      = {1988}
}

@article{alstrup.kaplan.ea:adjacency,
  author    = {Stephen Alstrup and
               Haim Kaplan and
               Mikkel Thorup and
               Uri Zwick},
  title     = {Adjacency Labeling Schemes and Induced-Universal Graphs},
  journal   = {{SIAM} J. Discrete Math.},
  volume    = {33},
  number    = {1},
  pages     = {116--137},
  year      = {2019},
  url       = {https://doi.org/10.1137/16M1105967},
  doi       = {10.1137/16M1105967},
  bibsource = {dblp computer science bibliography, https://dblp.org}
}

@article{alstrup.dahlgaard.ea:optimal,
  author    = {Stephen Alstrup and
               S{\o}ren Dahlgaard and
               Mathias B{\ae}k Tejs Knudsen},
  title     = {Optimal Induced Universal Graphs and Adjacency Labeling for Trees},
  journal   = {J. {ACM}},
  volume    = {64},
  number    = {4},
  pages     = {27:1--27:22},
  year      = {2017},
  url       = {https://doi.org/10.1145/3088513},
  doi       = {10.1145/3088513},
  bibsource = {dblp computer science bibliography, https://dblp.org}
}

@inproceedings{abrahamsen.alstrup.ea:near-optimal,
  author    = {Mikkel Abrahamsen and
               Stephen Alstrup and
               Jacob Holm and
               Mathias B{\ae}k Tejs Knudsen and
               Morten St{\"{o}}ckel},
  editor    = {Ioannis Chatzigiannakis and
               Piotr Indyk and
               Fabian Kuhn and
               Anca Muscholl},
  title     = {Near-Optimal Induced Universal Graphs for Bounded Degree Graphs},
  booktitle = {44th International Colloquium on Automata, Languages, and Programming,
               {ICALP} 2017, July 10-14, 2017, Warsaw, Poland},
  series    = {LIPIcs},
  volume    = {80},
  pages     = {128:1--128:14},
  publisher = {Schloss Dagstuhl - Leibniz-Zentrum f{\"{u}}r Informatik},
  year      = {2017},
  url       = {https://doi.org/10.4230/LIPIcs.ICALP.2017.128},
  doi       = {10.4230/LIPIcs.ICALP.2017.128},
  bibsource = {dblp computer science bibliography, https://dblp.org}
}

@article{AlonCapalbo,
author = {Alon, Noga and Capalbo, Michael},
title = {Sparse universal graphs for bounded-degree graphs},
journal = {Random Structures \& Algorithms},
volume = {31},
number = {2},
pages = {123--133},
doi = {10.1002/rsa.20143},
url = {https://onlinelibrary.wiley.com/doi/abs/10.1002/rsa.20143},
eprint = {https://onlinelibrary.wiley.com/doi/pdf/10.1002/rsa.20143},
year = {2007}
}

@article{Val81,
author = {Valiant, Leslie G.},
title = {Universality Considerations in {VLSI} Circuits},
year = {1981},
publisher = {IEEE Computer Society},
address = {USA},
volume = {30},
number = {2},
issn = {0018-9340},
journal = {IEEE Transactions on Computers},
pages = {135--140},
numpages = {6}
}

@article{CRS83,
author = {Chung, Fan R. K. and Rosenberg, Arnold L. and Snyder, Lawrence},
title = {Perfect Storage Representations for Families of Data Structures},
journal = {SIAM Journal on Algebraic Discrete Methods},
volume = {4},
number = {4},
pages = {548-565},
year = {1983},
doi = {10.1137/0604055},
URL = { https://doi.org/10.1137/0604055},
eprint = {https://doi.org/10.1137/0604055}
}

@inproceedings{bhatt1986optimal,
  author    = {Sandeep N. Bhatt and
               Fan R. K. Chung and
               Frank T. Leighton and
               Arnold L. Rosenberg},
  title     = {Optimal Simulations of Tree Machines (Preliminary Version)},
  booktitle = {27th Annual Symposium on Foundations of Computer Science, Toronto,
               Canada, 27-29 October 1986},
  pages     = {274--282},
  publisher = {{IEEE} Computer Society},
  year      = {1986},
  url       = {https://doi.org/10.1109/SFCS.1986.38},
  doi       = {10.1109/SFCS.1986.38},
  bibsource = {dblp computer science bibliography, https://dblp.org}
}

@ARTICLE{Capalbo02,
     AUTHOR = {Capalbo, Michael R.},
      TITLE = {Small Universal Graphs for Bounded-Degree Planar Graphs},
    JOURNAL = {Combinatorica},
     VOLUME = 22,
     NUMBER = 3,
      PAGES = {345--359},
       YEAR = 2002,
        DOI = {10.1007/s004930200017}
}

@inproceedings{ACKRRS,
  author    = {Noga Alon and
               Michael R. Capalbo and
               Yoshiharu Kohayakawa and
               Vojtech R{\"{o}}dl and
               Andrzej Rucinski and
               Endre Szemer{\'{e}}di},
  editor    = {Michel X. Goemans and
               Klaus Jansen and
               Jos{\'{e}} D. P. Rolim and
               Luca Trevisan},
  title     = {Near-optimum Universal Graphs for Graphs with Bounded Degrees},
  booktitle = {Approximation, Randomization and Combinatorial Optimization: Algorithms
               and Techniques, 4th International Workshop on Approximation Algorithms
               for Combinatorial Optimization Problems, {APPROX} 2001 and 5th International
               Workshop on Randomization and Approximation Techniques in Computer
               Science, {RANDOM} 2001 Berkeley, CA, USA, August 18-20, 2001, Proceedings},
  series    = {Lecture Notes in Computer Science},
  volume    = {2129},
  pages     = {170--180},
  publisher = {Springer},
  year      = {2001},
  url       = {https://doi.org/10.1007/3-540-44666-4\_20},
  doi       = {10.1007/3-540-44666-4\_20},
  bibsource = {dblp computer science bibliography, https://dblp.org}
}

@inproceedings{AC08,
author = {Alon, Noga and Capalbo, Michael},
title = {Optimal Universal Graphs with Deterministic Embedding},
year = {2008},
publisher = {Society for Industrial and Applied Mathematics},
address = {USA},
booktitle = {Proceedings of the Nineteenth Annual ACM-SIAM Symposium on Discrete Algorithms},
pages = {373–378},
numpages = {6},
location = {San Francisco, California},
series = {SODA '08}
}

@Article{EJM23,
 Author = {Esperet, Louis and Joret, Gwena{\"e}l and Morin, Pat},
 Title = {Sparse universal graphs for planarity},
 FJournal = {Journal of the London Mathematical Society. Second Series},
 Journal = {J. Lond. Math. Soc., II. Ser.},
 ISSN = {0024-6107},
 Volume = {108},
 Number = {4},
 Pages = {1333--1357},
 Year = {2023},
 Language = {English},
 DOI = {10.1112/jlms.12781},
 URL = {dipot.ulb.ac.be/dspace/bitstream/2013/360021/3/2010.05779.pdf},
 zbMATH = {7773459},
 Zbl = {1527.05041}
}

@inproceedings{HH21,
  title={The implicit graph conjecture is false},
  author={Hatami, Hamed and Hatami, Pooya},
  booktitle={63rd IEEE Symposium on Foundations of Computer Science (FOCS 2022)},
  year={2022},
  pages={1134-1137},
  doi={10.1109/FOCS54457.2022.00109}
}

@Article{AN19,
 Author = {Alon, Noga and Nenadov, Rajko},
 Title = {Optimal induced universal graphs for bounded-degree graphs},
 FJournal = {Mathematical Proceedings of the Cambridge Philosophical Society},
 Journal = {Math. Proc. Camb. Philos. Soc.},
 ISSN = {0305-0041},
 Volume = {166},
 Number = {1},
 Pages = {61--74},
 Year = {2019},
 Language = {English},
 DOI = {10.1017/S0305004117000706},
 zbMATH = {7009204},
 Zbl = {1404.05032}
}

@article{BNMST24,
  author       = {{\'{E}}douard Bonnet and
                  Jaroslav Ne\v{s}et\v{r}il and
                  Patrice Ossona de Mendez and
                  Sebastian Siebertz and
                  St{\'{e}}phan Thomass{\'{e}}},
  title        = {Twin-width and permutations},
  journal      = {Log. Methods Comput. Sci.},
  volume       = {20},
  number       = {3},
  year         = {2024},
  url          = {https://doi.org/10.46298/lmcs-20(3:4)2024},
  doi          = {10.46298/LMCS-20(3:4)2024},
  bibsource    = {dblp computer science bibliography, https://dblp.org}
}

@Article{Bol82,
 Author = {Bollob\'as, Bel\'a},
 Title = {The asymptotic number of unlabelled regular graphs},
 FJournal = {Journal of the London Mathematical Society. Second Series},
 Journal = {J. Lond. Math. Soc., II. Ser.},
 ISSN = {0024-6107},
 Volume = {26},
 Pages = {201--206},
 Year = {1982},
 Language = {English},
 DOI = {10.1112/jlms/s2-26.2.201},
 zbMATH = {3794116},
 Zbl = {0504.05051}
}

@Article{Bol80,
 Author = {Bollob\'as, Bel\'a},
 Title = {A probabilistic proof of an asymptotic formula for the number of labelled regular graphs},
 FJournal = {European Journal of Combinatorics},
 Journal = {Eur. J. Comb.},
 ISSN = {0195-6698},
 Volume = {1},
 Pages = {311--316},
 Year = {1980},
 Language = {English},
 DOI = {10.1016/S0195-6698(80)80030-8},
 zbMATH = {3715608},
 Zbl = {0457.05038}
}

@article{CGS81,
title = {Universal caterpillars},
journal = {Journal of Combinatorial Theory, Series B},
volume = {31},
number = {3},
pages = {348-355},
year = {1981},
issn = {0095-8956},
doi = {https://doi.org/10.1016/0095-8956(81)90037-X},
url = {https://www.sciencedirect.com/science/article/pii/009589568190037X},
author = {Fan R. K. Chung and Ronald L. Graham and James B. Shearer}
}

@Article{Kos84,
 Author = {Kostochka, Alexandr V.},
 Title = {Lower bound of the {Hadwiger} number of graphs by their average degree},
 FJournal = {Combinatorica},
 Journal = {Combinatorica},
 ISSN = {0209-9683},
 Volume = {4},
 Pages = {307--316},
 Year = {1984},
 Language = {English},
 DOI = {10.1007/BF02579141},
 zbMATH = {3884189},
 Zbl = {0555.05030}
}

@Article{Tho84,
 Author = {Thomason, Andrew},
 Title = {An extremal function for contractions of graphs},
 FJournal = {Mathematical Proceedings of the Cambridge Philosophical Society},
 Journal = {Math. Proc. Camb. Philos. Soc.},
 ISSN = {0305-0041},
 Volume = {95},
 Pages = {261--265},
 Year = {1984},
 Language = {English},
 DOI = {10.1017/S0305004100061521},
 zbMATH = {3877225},
 Zbl = {0551.05047}
}

@Article{Sau72,
 Author = {Sauer, Norbert W.},
 Title = {On the density of families of sets},
 FJournal = {Journal of Combinatorial Theory. Series A},
 Journal = {J. Comb. Theory, Ser. A},
 ISSN = {0097-3165},
 Volume = {13},
 Pages = {145--147},
 Year = {1972},
 Language = {English},
 DOI = {10.1016/0097-3165(72)90019-2},
 zbMATH = {3392460},
 Zbl = {0248.05005}
}

@Article{She72,
 Author = {Shelah, Saharon},
 Title = {A combinatorial problem; stability and order for models and theories in infinitary languages},
 FJournal = {Pacific Journal of Mathematics},
 Journal = {Pac. J. Math.},
 ISSN = {1945-5844},
 Volume = {41},
 Pages = {247--261},
 Year = {1972},
 Language = {English},
 DOI = {10.2140/pjm.1972.41.247},
 zbMATH = {3378905},
 Zbl = {0239.02024}
}

@article{Kawarabayashi2018,
  title = {A new proof of the flat wall theorem},
  volume = {129},
  ISSN = {0095-8956},
  url = {http://dx.doi.org/10.1016/j.jctb.2017.09.006},
  DOI = {10.1016/j.jctb.2017.09.006},
  journal = {Journal of Combinatorial Theory,  Series B},
  publisher = {Elsevier BV},
  author = {Kawarabayashi,  Ken-ichi and Thomas,  Robin and Wollan,  Paul},
  year = {2018},
  pages = {204–238}
}

@article{ackermann1937widerspruchsfreiheit,
  title={Die widerspruchsfreiheit der allgemeinen mengenlehre},
  author={Ackermann, Wilhelm},
  journal={Mathematische Annalen},
  volume={114},
  number={1},
  pages={305--315},
  year={1937},
  publisher={Springer}
}

@article{erdos1963asymmetric,
  title={Asymmetric graphs},
  author={Erd\H{o}s, Paul and R{\'e}nyi, Alfr{\'e}d},
  journal={Acta Math. Acad. Sci. Hungar},
  volume={14},
  number={295-315},
  pages={3},
  year={1963}
}

@article{cameron1997random,
  title={The random graph},
  author={Cameron, Peter J.},
  journal={The Mathematics of Paul Erd{\H{o}}s II},
  pages={333--351},
  year={1997},
  publisher={Springer}
}

@article{henson1971family,
  title={A family of countable homogeneous graphs},
  author={Henson, C. Ward},
  journal={Pacific journal of mathematics},
  volume={38},
  number={1},
  pages={69--83},
  year={1971},
  publisher={Mathematical Sciences Publishers}
}

@inproceedings{cameron2001random,
  title={The random graph revisited},
  author={Cameron, Peter J.},
  booktitle={European Congress of Mathematics: Barcelona, July 10--14, 2000, Volume I},
  pages={267--274},
  year={2001},
  organization={Springer}
}

@article{rado1964universal,
  title={Universal graphs and universal functions},
  author={Rado, Richard},
  journal={Acta Arithmetica},
  volume={9},
  pages={331--340},
  year={1964},
  publisher={Instytut Matematyczny Polskiej Akademii Nauk}
}

@article{cherlin2007universalforbidden,
  title={Universal graphs with a forbidden subtree},
  author={Cherlin, Gregory and Shelah, Saharon},
  journal={Journal of Combinatorial Theory, Series B},
  volume={97},
  number={3},
  pages={293--333},
  year={2007},
  publisher={Elsevier}
}

@article{cherlin2016universal,
  title={Universal graphs with a forbidden subgraph: block path solidity},
  author={Cherlin, Gregory and Shelah, Saharon},
  journal={Combinatorica},
  volume={36},
  number={3},
  pages={249--264},
  year={2016},
  publisher={Springer}
}

@article{cherlin1999universal,
  title={Universal graphs with forbidden subgraphs and algebraic closure},
  author={Cherlin, Gregory and Shelah, Saharon and Shi, Niandong},
  journal={Advances in Applied Mathematics},
  volume={22},
  number={4},
  pages={454--491},
  year={1999},
  publisher={Elsevier}
}

@article{cherlin2001forbidden,
  title={Forbidden subgraphs and forbidden substructures},
  author={Cherlin, Gregory and Shi, Niandong},
  journal={The Journal of Symbolic Logic},
  volume={66},
  number={3},
  pages={1342--1352},
  year={2001},
  publisher={Cambridge University Press}
}

@article{cherlin2007universal,
  title={Universal graphs with a forbidden near-path or 2-bouquet},
  author={Cherlin, Gregory and Tallgren, Lasse},
  journal={Journal of Graph Theory},
  volume={56},
  number={1},
  pages={41--63},
  year={2007},
  publisher={Wiley Online Library}
}

@article{furedi1997nonexistence,
  title={Nonexistence of universal graphs without some trees},
  author={F{\"u}redi, Zoltan and Komj{\'a}th, P{\'e}ter},
  journal={Combinatorica},
  volume={17},
  pages={163--171},
  year={1997},
  publisher={Springer}
}

@article{furedi1997existence,
  title={On the existence of countable universal graphs},
  author={F{\"u}redi, Zoltan and Komj{\'a}th, P{\'e}ter},
  journal={Journal of Graph Theory},
  volume={25},
  number={1},
  pages={53--58},
  year={1997},
  publisher={Wiley Online Library}
}

@article{komjath1999some,
  title={Some remarks on universal graphs},
  author={Komj{\'a}th, P{\'e}ter},
  journal={Discrete mathematics},
  volume={199},
  number={1-3},
  pages={259--265},
  year={1999},
  publisher={Elsevier}
}

@article{komjath1988some,
  title={Some universal graphs},
  author={Komj{\'a}th, P{\'e}ter and Mekler, Alan H and Pach, J{\'a}nos},
  journal={Israel Journal of Mathematics},
  volume={64},
  number={2},
  pages={158--168},
  year={1988},
  publisher={Springer}
}

@article{komjath1984universal,
  title={Universal graphs without large bipartite subgraphs},
  author={Komj{\'a}th, P{\'e}ter and Pach, J{\'a}nos},
  journal={Mathematika},
  volume={31},
  number={2},
  pages={282--290},
  year={1984},
  publisher={Wiley Online Library}
}

@article{diestel1985universal,
  title={On universal graphs with forbidden topological subgraphs},
  author={Diestel, Reinhard},
  journal={European Journal of Combinatorics},
  volume={6},
  number={2},
  pages={175--182},
  year={1985},
  publisher={Elsevier}
}

@article{pach1981problem,
  title={A Problem of {U}lam on Planar Graphs.},
  author={Pach, J{\'a}nos},
  journal={Eur. J. Comb.},
  volume={2},
  number={4},
  pages={357--361},
  year={1981}
}

@article{HW25,
      title={Short Paths in the Planar Graph Product Structure Theorem}, 
      author={Kevin Hendrey and David R. Wood},
      year={2025},
      journal={arXiv preprint},
      eprint={2502.01927},
      archivePrefix={arXiv},
      primaryClass={math.CO},
      note={\href{https://arxiv.org/abs/2502.01927}{arXiv:2502.01927}}, 
}

@Article{Sey80,
 Author = {Seymour, Paul D.},
 Title = {Disjoint paths in graphs},
 FJournal = {Discrete Mathematics},
 Journal = {Discrete Math.},
 ISSN = {0012-365X},
 Volume = {29},
 Pages = {293--309},
 Year = {1980},
 Language = {English},
 DOI = {10.1016/0012-365X(80)90158-2},
 zbMATH = {3715613},
 Zbl = {0457.05043}
}

@Article{Tho80,
 Author = {Thomassen, Carsten},
 Title = {2-linked graphs},
 FJournal = {European Journal of Combinatorics},
 Journal = {Eur. J. Comb.},
 ISSN = {0195-6698},
 Volume = {1},
 Pages = {371--378},
 Year = {1980},
 Language = {English},
 DOI = {10.1016/S0195-6698(80)80039-4},
 zbMATH = {3715614},
 Zbl = {0457.05044}
}

@article{BEFZ25,
title = {Renaming in distributed certification},
journal = {Theoretical Computer Science},
volume = {1061},
pages = {115643},
year = {2026},
issn = {0304-3975},
doi = {https://doi.org/10.1016/j.tcs.2025.115643},
url = {https://www.sciencedirect.com/science/article/pii/S0304397525005808},
author = {Nicolas Bousquet and Louis Esperet and Laurent Feuilloley and Sébastien Zeitoun},
}

@inproceedings{HWY10,
  author       = {Pavel Hrube{\v{s}} and
                  Avi Wigderson and
                  Amir Yehudayoff},
  title        = {Relationless Completeness and Separations},
  booktitle    = {Proceedings of the 25th Annual {IEEE} Conference on Computational
                  Complexity, {CCC} 2010, Cambridge, Massachusetts, USA, June 9-12,
                  2010},
  pages        = {280--290},
  publisher    = {{IEEE} Computer Society},
  year         = {2010},
  url          = {https://doi.org/10.1109/CCC.2010.34},
  doi          = {10.1109/CCC.2010.34},
  bibsource    = {dblp computer science bibliography, https://dblp.org}
}

@book{Stan15,
 author = {Stanley, Richard P.},
 title = {Catalan numbers},
 isbn = {978-1-107-07509-2; 978-1-107-42774-7; 978-1-139-87149-5},
 year = {2015},
 publisher = {Cambridge: Cambridge University Press},
 language = {English},
 doi = {10.1017/CBO9781139871495},
 zbMATH = {6417735},
 Zbl = {1317.05010}
}

@article{Wein66,
 author = {Weinberg, Louis},
 title = {On the maximum order of the automorphism group of a planar triply connected graph},
 fjournal = {SIAM Journal on Applied Mathematics},
 journal = {SIAM J. Appl. Math.},
 issn = {0036-1399},
 volume = {14},
 pages = {729--738},
 year = {1966},
 language = {English},
 doi = {10.1137/0114062},
 zbMATH = {3234261},
 Zbl = {0145.20601}
}

@article{HT66,
 author = {Harary, Frank and Tutte, William T.},
 title = {On the order of the group of a planar map},
 fjournal = {Journal of Combinatorial Theory},
 journal = {J. Comb. Theory},
 issn = {0097-3165},
 volume = {1},
 pages = {394--395},
 year = {1966},
 language = {English},
 doi = {10.1016/S0021-9800(66)80060-1},
 zbMATH = {3234262},
 Zbl = {0145.20602}
}

@article{KKKW25,
      title={On Universal Graphs for Trees and Tree-Like Graphs}, 
      author={Neel Kaul and Jaehoon Kim and Minseo Kim and David R. Wood},
      year={2026},
      eprint={2511.22358},
      archivePrefix={arXiv},
      journal={arXiv preprint},
      primaryClass={math.CO},
      note={\href{https://arxiv.org/abs/2511.22358}{arXiv:2511.22358}}, 
}

@article{FHT23,
 author = {Frati, Fabrizio and Hoffmann, Michael and T{\'o}th, Csaba D.},
 title = {Universal geometric graphs},
 fjournal = {Combinatorics, Probability and Computing},
 journal = {Comb. Probab. Comput.},
 issn = {0963-5483},
 volume = {32},
 number = {5},
 pages = {742--761},
 year = {2023},
 language = {English},
 doi = {10.1017/S0963548323000135},
 zbMATH = {7938693},
 Zbl = {1551.05082}
}

@article{BP11,
 author = {Balogh, J{\'o}zsef and Pet{\v{r}}{\'{\i}}{\v{c}}kov{\'a}, {\v{S}}{\'a}rka},
 title = {The number of the maximal triangle-free graphs},
 fjournal = {Bulletin of the London Mathematical Society},
 journal = {Bull. Lond. Math. Soc.},
 issn = {0024-6093},
 volume = {46},
 number = {5},
 pages = {1003--1006},
 year = {2014},
 language = {English},
 doi = {10.1112/blms/bdu059},
 zbMATH = {6355598},
 Zbl = {1302.05087}
}

\appendix

\section{Proof of Lemma~\ref{Kt grid one interior bis}}\label{sec:app}

We say that a jump between vertices $(x,y)$ and $(x',y')$ in a triangulated grid is \emph{local} if $|x-x'| = |y - y'| = 1$.
We will need the following lemma from \cite{Kawarabayashi2018} which creates a $K_t$-minor out of local jumps that are all contained within two consecutive internal rows (note that the statement below is slightly weaker than in \cite{Kawarabayashi2018}).

\begin{lemma}[{Lemma~3.2 in~\cite{Kawarabayashi2018}}]
    \label{Kt grid middle line crossings bis}
    There is a polynomial function $f_{\ref{Kt grid middle line crossings bis}}\colon \mathbb{N}_{>0} \to \mathbb{N}_{>0}$ such that the following holds.
    Let $t, \ell,\ell'$ be positive integers    
    with $\ell,\ell' \geq 2f_{\ref{Kt grid middle line crossings bis}}(t)$.

    Let $G$ be a triangulated $\ell \times \ell'$ grid,
    and let $R_{i}$, $R_{i+1}$ be two $f_{\ref{Kt grid middle line crossings bis}}(t)$-internal consecutive rows of $G$. 
    For every set $M$ of pairwise disjoint jumps of $G$,
    if
    \begin{enumerate}
        \item for every $uv \in M$, $uv$ is local and $u,v \in R_i \cup R_{i+1}$; and
        \item $|M| \geq f_{\ref{Kt grid middle line crossings bis}}(t)$; 
    \end{enumerate}
    then $K_t$ is a minor of $G \cup M$.
\end{lemma}

We obtain the following as a simple consequence, where the local jumps are allowed to be scattered in the grid, instead of all concentrated on a central row.

\begin{lemma}
    \label{Kt when a lot of local edges}
    There is a polynomial function $f_{\ref{Kt when a lot of local edges}}\colon \mathbb{N}_{>0} \to \mathbb{N}_{>0}$ such that the following holds.
    Let $t, \ell,\ell'$ be positive integers with $\ell,\ell' \geq 2f_{\ref{Kt when a lot of local edges}}(t)$. 
    Let $G$ be a triangulated $\ell \times \ell'$ grid.
    For every set $M$ of pairwise disjoint jumps of $G$,
    if
    \begin{enumerate}
        \item for every $uv \in M$, $uv$ is local and $u, v$ are $f_{\ref{Kt when a lot of local edges}}(t)$-internal; and
        \item $|M| \geq f_{\ref{Kt when a lot of local edges}}(t)$; 
    \end{enumerate}
    then $K_t$ is a minor of $G \cup M$.
\end{lemma}

\begin{proof}
    Let
    \[
        f_{\ref{Kt when a lot of local edges}}(t) = 1+2(f_{\ref{Kt grid middle line crossings bis}}(t) - 1) f_{\ref{Kt grid middle line crossings bis}}(t).
    \]
    Let $j_1=i_1=f_{\ref{Kt when a lot of local edges}}(t)$, $i_2=\ell'- f_{\ref{Kt when a lot of local edges}}(t)$ and $j_2=\ell-f_{\ref{Kt when a lot of local edges}}$,
    which are the indices of the first and last internal rows and columns.
    First suppose that there exists $j \in \{j_1, \dots ,j_2-1\}$ such that 
    the family $M$ contains at least $f_{\ref{Kt grid middle line crossings bis}}(t)$ pairs of vertices in $C_j\cup C_{j+1}$ (recall that for any $i$, $C_i$ denotes the $i$-th column). 
    Then, by \Cref{Kt grid middle line crossings bis}, $G \cup M$ contains $K_t$ as a minor (since all pairs in $M$ are local by assumption). (Strictly speaking, the lemma was stated for rows instead of columns, but the column-version holds as well by symmetry.)
    Now suppose that no such $j$ exists.

    For every $e \in M$, let $j_e \in \{j_1, \dots, j_2-1\}$ be such that 
    $e \subseteq C_{j_e} \cup C_{j_e+1}$.
    By the previous paragraph, there is no $j$ for which $j_e=j$ for at least $f_{\ref{Kt grid middle line crossings bis}}(t)$ elements $e\in M$. 
    By the pigeonhole principle, there exists $M_1 \subseteq M$ of size $2f_{\ref{Kt grid middle line crossings bis}}(t)$ such that
    the indices $j_e$ for $e \in M_1$ are pairwise distinct.
    
    Since the graph with vertex set $M_1$ and edges all the pairs $e,e' \in M_1$ with  $|j_e - j_{e'}| = 1$ is a disjoint union of paths, it has a stable set $M_2$ of size $\frac{|M_1|}{2} = f_{\ref{Kt grid middle line crossings bis}}(t)$.
    
    Let $j \in [\ell]$.
    By definition of $M_2$, at most one pair $e \in M_2$ intersects the $j^{\text{th}}$ column $C_j$. 
    It may be helpful to look at Figure~\ref{fig:a lot of local edges} while reading the notation below.
    \begin{figure}
        \centering
        \includegraphics{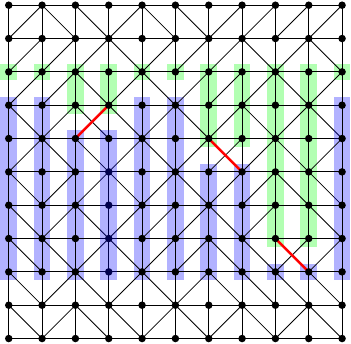}
        \caption{Illustration for the proof of \Cref{Kt when a lot of local edges}.
            In red are depicted the local pairs in $M_2$.
            In respectively green and blue are the set $D_j$ and $U_j$ for $j \in [\ell]$.
            Informally, contracting these sets moves the crossings to the middle row.} 
        \label{fig:a lot of local edges}
    \end{figure}
    If there exists $e\in M_2$ intersecting $C_j$, then let $y<y'$ and $x,x'\in \{j-1,j,j+1\}$ such that $e=(x,y)(x',y')$. We set
    \begin{align*}
        D_j &= \{(j,i) \mid i_1 \leq i \leq y\} \\
        U_j &= \{(j,i) \mid y' \leq i \leq i_2\}.
    \end{align*}
    If no such $e$ exists, let
    \begin{align*}
        D_j &= \{(j,i_1)\} \\
        U_j &= \{(j,i) \mid i_1+1 \leq i \leq i_2\}.
    \end{align*}
    In particular, $D_j$ denotes the ``down''-part of column $j$ and $U_j$ the ``upper'' part.
    
    Every edge of $M_2$ has one endpoint in $U_j$ and the other in $D_{j'}$ (for some $j,j'$).
    Let $G'$ be the graph obtained from $G$ by contracting $D_j$ and $U_j$ into single vertices for every $j \in [\ell']$. 
    Then $G'$ is isomorphic to a triangulated $\ell \times (\ell'-(i_2-i_1-1))$ grid. 
    Moreover, the set $M_2$ induces in $G'$ a family $M'$ of pairwise disjoint local pairs of $G'$ 
    not in $E(G')$, each of them included in the union of the $i_1^{\text{th}}$ and $(i_1+1)^{\text{th}}$ rows of $G'$.
    Furthermore, by assumption, these pairs consist of $f_{\ref{Kt when a lot of local edges}}(t)$-internal vertices. 
    Since $|M_2| \geq f_{\ref{Kt grid middle line crossings bis}}(t)$, and $f_{\ref{Kt when a lot of local edges}}(t) \geq f_{\ref{Kt grid middle line crossings bis}}(t)$
    we deduce by \Cref{Kt grid middle line crossings bis} that $G' \cup M'$ contains $K_t$ as a minor,
    and so $G \cup M$ does too.
    This proves the lemma. 
\end{proof}

We now extend \Cref{Kt when a lot of local edges} to the case where for all jumps  $uv\in M$, $u$ and $v$ are close in the grid.

\begin{lemma}
    \label{lem:slocal}
    There is a polynomial function $f_{\ref{lem:slocal}}\colon \mathbb{N}_{>0} \to \mathbb{N}_{>0}$ such that the following holds.
    Let $t, \ell,\ell'$ be positive integers with $\ell,\ell' \geq 2f_{\ref{lem:slocal}}(t)$. Let $G$ be a triangulated $\ell \times \ell'$ grid.
    For every set $M$ of pairwise disjoint jumps of $G$,
    if
    \begin{enumerate}
        \item for every $uv \in M$, $d_G(u,v)\le 20t^2$ and $u,v$ are $f_{\ref{lem:slocal}}(t)$-internal; and
        \item $|M| \geq f_{\ref{lem:slocal}}(t)$; 
    \end{enumerate}
    then $G \cup M$ contains $K_t$ as a minor.
\end{lemma}

\begin{proof}
    Consider a subset of columns $\mathcal{C}=C_{i_1}, \ldots,C_{i_a}$ with $1=i_1<\cdots<i_a=\ell$ and a subset of rows $\mathcal{R}=R_{j_1}, \ldots,R_{j_b}$ with $1=j_1<\cdots<j_b=\ell'$. Note that the corresponding vertices in $G$ induce a subdivision $H$ of the $a\times b$ grid. Call the vertices of $G$ in the intersection of one of these rows and one of these columns \emph{branch vertices}. Note that we can find disjoint connected subgraphs (call them \emph{bags}) in $G$, each containing a different branch vertex, such that after contracting each of these bags into a single vertex, the resulting graph is a triangulated $a \times b$ grid $G'$. Moreover, it follows from the 2-linkage theorem of Seymour \cite{Sey80} and Thomassen \cite{Tho80} that for any jump $uv\in M$ lying strictly between two consecutive columns of $\mathcal{C}$ and two consecutive rows of $\mathcal{R}$,  $uv$ becomes a local jump in $G'$ (in the sense that we can choose the bags in such a way that $u$ and $v$ are included in two distinct and non-adjacent bags). 
    
    This shows that if we can find  a subset $\mathcal{C}=C_{i_1}, \ldots,C_{i_a}$ of $a$ columns and a subset $\mathcal{R}=R_{j_1}, \ldots,R_{j_b}$ of $b$ rows such that at least $s$ rectangular open regions bounded by consecutive columns $C_{i_c}, C_{i_{c+1}}$ and consecutive rows $R_{i_d}, R_{i_{d+1}}$ contain a jump of $M$, then $G\cup M$ contains as a minor a triangulated $a\times b$ grid $G'$ together with the edges corresponding to a set $M'$ of at least $s$ local jumps in $G'$. 
    Moreover if for each of the $s$ jumps as above, we have $k\le c \le a-k$  and $k\le d \le b-k$, then the corresponding local jumps of $M'$ are $k$-internal in $G'$. This is illustrated in Figure \ref{fig:localgrid}. If $a,b,s,k\ge 2f_{\ref{Kt when a lot of local edges}}(t)$ it follows from \Cref{Kt when a lot of local edges} that $G'\cup M'$ contains $K_t$ as a minor, and thus $G\cup M$ also contains $K_t$ as a minor. So we only need to find sets $\mathcal{R}$ and $\mathcal{C}$ of rows and columns as above.

    \medskip

    \begin{figure}[htb]
    \centering
    \includegraphics{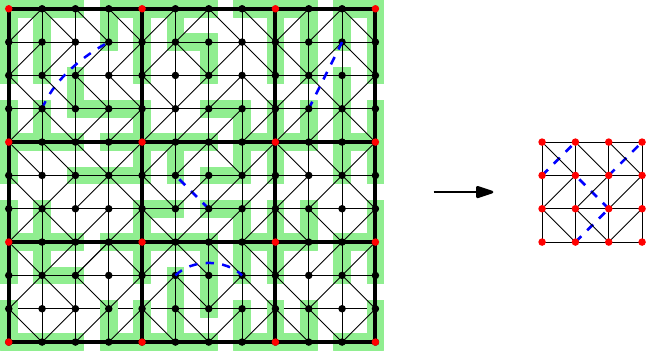}
    \caption{The rows of $\mathcal{R}$ and columns of $\mathcal{C}$ are in bold. Branch vertices are in red, and bags are depicted in light green. Jumps are depicted with dashed blue lines.}
    \label{fig:localgrid}
\end{figure}

    We set $f_{\ref{lem:slocal}}(t)=\tfrac{10}3(100 t^2)^2\cdot f_{\ref{Kt when a lot of local edges}}(t)$. We start by including in $\mathcal{R}$ the first and last $f_{\ref{Kt when a lot of local edges}}(t)$ rows of $G$, and we include in $\mathcal{C}$ the first and last $f_{\ref{Kt when a lot of local edges}}(t)$  columns of $G$. In addition, we choose an integer $p\in [100t^2]$ uniformly at random, we include in $\mathcal{R}$ each row whose index modulo $100t^2$ is equal to $p$ and in $\mathcal{C}$ each column of $G$ whose index modulo $100t^2$ is equal to $p$. Consider a jump $uv\in M$. As $d_G(u,v)\le 20t^2$, the probability that $u$ and $v$ are separated by a row of $\mathcal{R}$ or a column of $\mathcal{C}$ is at most $2\cdot \tfrac1{100t^2}\cdot 20t^2\le \tfrac25$. It follows that in expectation, $\tfrac35\,|M|$ jumps of $M$ are contained within two consecutive rows and columns of $\mathcal{R}$ and $\mathcal{C}$. Each region bounded by two consecutive rows and columns contains at most $(100t^2)^2$ vertices by definition, and since the jumps of $M$ are pairwise disjoint, such a region contains at most $(100t^2)^2$ jumps of $M$. It follows that there is a choice of $\mathcal{R}$ and $\mathcal{C}$ such that at least \[\tfrac35\,|M|\cdot \tfrac1{(100t^2)^2}\ge 2f_{\ref{Kt when a lot of local edges}}(t)\] regions bounded by consecutive rows and columns contain a jump of $M$. Since we have included in $\mathcal{C}$ and $\mathcal{R}$ the first and last $f_{\ref{Kt when a lot of local edges}}(t)$ rows and columns of $G$, the jumps we extract are all $f_{\ref{Kt when a lot of local edges}}(t)$-internal, as desired.
\end{proof}

It remains to consider the case where for many jumps $uv\in M$, $u$ and $v$ are far apart in the grid. The result below was proved in a slightly stronger form in \cite{Kawarabayashi2018}.

\begin{lemma}[Lemma 4.5 in \cite{Kawarabayashi2018}]
    \label{lem:nonslocal}
    There is a polynomial function $f_{\ref{lem:nonslocal}}\colon \mathbb{N}_{>0} \to \mathbb{N}_{>0}$ such that the following holds.
    Let $t, \ell,\ell'$ be positive integers with $\ell,\ell' \geq 2f_{\ref{lem:nonslocal}}(t)$.
    Let $G$ be a triangulated $\ell \times \ell'$ grid.
    For every set $M$ of pairwise disjoint jumps of $G$,
    if
    \begin{enumerate}
        \item for every $uv \in M$, at least one of $u,v$ is  $20t^2$-internal; 
        \item all the vertices that are an endpoint of a jump from $M$ lie pairwise at distance  more than $20t^2$ in $G$; 
        \item $|M| \geq f_{\ref{lem:nonslocal}}(t)$; 
    \end{enumerate}
    then $K_t$ is a minor of $G \cup M$.
\end{lemma}

We are now ready to prove the main  result of this section,  \Cref{Kt grid one interior bis}, which we restate below for convenience.

\ktgridoneinteriorbis*

\begin{proof}
Let $f_{\ref{Kt grid one interior bis}}(t)=f_{\ref{lem:slocal}}(t)+(2(40t^2+1)^2+1)f_{\ref{lem:nonslocal}}(t)+20t^2$.

Assume first that $M$  contains at least $f_{\ref{lem:slocal}}(t)$ jumps $uv$ such that $d_G(u,v)\le 20t^2$. Note that for such a pair $uv$, 
if at least one of $u,v$ is $s$-internal (for some constant $s$) then both $u$ and $v$ are $(s-20t^2)$-internal. It follows that all jumps  $uv\in M$ with  $d_G(u,v)\le 20t^2$ are $f_{\ref{lem:slocal}}(t)$-internal and thus \Cref{lem:slocal} implies that $K_t$ is a minor of $G \cup M$, as desired.

We can now assume that $M$  contains at least $(2(40t^2+1)^2+1)f_{\ref{lem:nonslocal}}(t)$ jumps $uv$ with $d_G(u,v)> 20t^2$. As $G$ is a triangulated grid, for every vertex $v\in V(G)$, the number of vertices of $G$ lying at distance at most $20t^2$ from $v$ in $G$ is at most $(40t^2+1)^2$. Since the jumps of $M$ are pairwise disjoint, it follows that every jump of $M$ lies at distance at most $20t^2$ from at most $2(40t^2+1)^2$ jumps of $M$. We can thus find a subset $M'\subseteq M$
 of $f_{\ref{lem:nonslocal}}(t)$ of jumps $uv\in M$ with $d_G(u,v)> 20t^2$, and such that all endpoints of the elements of $M'$  lie pairwise at distance more than $20t^2$ in $G$. We can now apply \Cref{lem:nonslocal}, which implies that $K_t$ is a minor of $G \cup M$. 
\end{proof}

\section{Treedepth definitions and proof of Corollary~\ref{td upper bound}}\label{sec:td}
In a rooted tree $T$, we say that $u$ is an \emph{ancestor} of $v$ if $u$ lies on the unique path between the root and $v$ in $T$. A \emph{rooted forest} is a disjoint union of rooted trees, and the ancestor relation naturally extends to rooted forests.
The \emph{transitive closure} of a rooted forest $F$ is the graph $(V(F), E(F)\cup E')$ obtained by adding to $F$ the set of edges $E' = \{uv\mid u \text{ is an ancestor of }v\}$.
An \emph{elimination forest} for a graph $G$ is a rooted forest $F$ such that $V(G) \subseteq V(F)$ and for every edge $uv\in E(G)$ there exists a root-to-leaf path $P$ in $F$ such that $u,v\in V(P)$. In other word, $G$ is a subgraph of the \emph{transitive closure} of $F$.
The \emph{depth} of a rooted forest $F$ is the maximum number of vertices over all root-to-leaf paths $P$ in $F$. 
The \emph{treedepth} of a graph $G$, denoted $\td(G)$, is the minimum depth over all the elimination forests of $G$.

\begin{remark}\label{rem:td} Note that we can always assume that there is an optimal elimination forest $F$ for $G$ such that $V(G)=V(F)$. To see this, it suffices to observe that if some root $r$ of a tree $T$ of $F$ is such that 
no vertex of $G$ is mapped to $r$, then we can replace $T$ by $T-r$ in $F$ (and the observation then follows by a repeated applications of this operation to the lower levels of the forests).
\end{remark}

We can observe that for any DFS-ordering  of the leaves of an elimination forest $F$ of $G$, the sequence of root-to-leaf paths in $F$ gives a path-decomposition of $G$, and thus $\tw(G)\le \pw(G)\le \td(G)-1$ for any graph $G$.

\medskip

Given a graph class $\mathcal{G}$, we denote by $\mathcal{G}_\mathrm{conn}$  the class of connected graphs from $\mathcal{G}$. We say that a graph class is \emph{summable} if the disjoint union of any two graphs in the class is still in the class (this property holds for all classes we consider in this paper). We say that a class is \emph{subtractable} if for any graph in the class, all its connected components are also in the class (this property holds in particular for any hereditary class). The next result will be useful to be able to restrict ourselves to connected graphs.

\begin{lemma}\label{lem:connected}
Let $\mathcal{G}$ be a graph class which is summable and subtractable. Assume that for any $n$, there is a graph $U_n\in \mathcal{G}$ that contains all $n$-vertex graphs from  $\mathcal{G}_\mathrm{conn}$ as subgraph (resp.\ induced subgraph). Then for every  $n$ there is a graph  $U_n'\in \mathcal{G}$ on at most $\sum_{i=0}^{\lceil \log_2 n \rceil} |V(U_{2^i})| \frac{n}{2^{i-1}}$ vertices that contains all $n$-vertex graphs from  $\mathcal{G}$ as subgraph (resp.\ induced subgraph).
\end{lemma}

\begin{proof}
Let $f(n)=|V(U_n)$|.  For a graph $G\in \mathcal{G}$ and $0\le i \le \lceil\log_2 n\rceil$,
the union of the connected components of $G$ of order between $2^{i-1}$ and $2^i-1$
appears as a subgraph (resp.\ induced subgraph) in the disjoint union of $\frac{n}{2^{i-1}}$ copies of the graph $ U_{2^i}$.
By taking the union over all $i$, we obtain a subgraph-universal (resp.\ induced-universal) graph of order 
$\sum_{i=0}^{\lceil \log_2 n \rceil} f(2^i) \frac{n}{2^{i-1}}$.
\end{proof}

Note that for $c>1$, $|V(U_n')|=O(n^c)$ if and only if $|V(U_n)|=O(n^c)$. Therefore, in this paper, we can focus on universal graphs for connected graphs.

\begin{lemma}\label{td upper bound2}
    For every $n,k\in \mathbb{N}$, the rooted tree $T_{n,k}$ is an elimination tree for every connected $n$-vertex graph of treedepth at most $k$.
\end{lemma}
\begin{proof}
    Let $G$ be a connected graph of treedepth at most $k$. Then there is a rooted tree $T$ of depth $k$ such that $T$ is an elimination tree of $G$ with $V(T)=V(G)$ (see Remark \ref{rem:td}). In particular, $T$ has at most $n$ vertices.
    Therefore, $T$ is a subtree of $T_{n,k}$.
    This implies that the transitive closure of $T$ is a subgraph of the transitive closure of $T_{n,k}$, and thus $G$ is a subgraph of the transitive closure of $T_{n,k}$.
    Therefore, $T_{n,k}$ is an elimination tree of $G$.
\end{proof}

Combining this lemma with \Cref{lem:connected} and the fact that $|V(T_{n,k})| = \frac{n^k-1}{n-1} \leq 2n^{k-1}$, we obtain  faithful subgraph-universal graphs of polynomial order for any class of graphs of bounded treedepth.

\end{document}